\documentclass[a4paper]{amsart}
\usepackage{graphicx}
\usepackage[all]{xy}
\theoremstyle{plain}
  \newtheorem{thm}{Theorem}[section]
  \newtheorem{lem}[thm]{Lemma}
  \newtheorem{cor}[thm]{Corollary}
  \newtheorem{prop}[thm]{Proposition}

\theoremstyle{definition}
  
  \newtheorem{ex}[thm]{Example}

\theoremstyle{remark}
  \newtheorem{rem}[thm]{Remark}

\newcommand{\Z}{\mathbb{Z}}
\newcommand{\cyclic}[1]{\Z/#1\Z}

\newcommand{\trivial}{\emptyset}
\newcommand{\trivialp}[1]{\trivial_{#1}}

\newcommand{\xType}[1]{\lvert#1\rvert}

\newcommand{\singular}[1]{\dot{#1}}
\newcommand{\subphrase}{\vartriangleleft}
\newcommand{\aplus}{a_+}
\newcommand{\aminus}{a_-}
\newcommand{\bplus}{b_+}
\newcommand{\bminus}{b_-}
\newcommand{\zeromatrix}{\bf{0}}
\DeclareMathOperator{\rank}{rank}

\numberwithin{equation}{section}
\allowdisplaybreaks
\begin{document}
\title{Finite type invariants of nanowords and nanophrases}
\author{Andrew Gibson}
\author{Noboru Ito}
\address{
Department of Mathematics,
Tokyo Institute of Technology,
Oh-okayama, Meguro, Tokyo 152-8551, Japan
}
\email{gibson@math.titech.ac.jp}
\address{
Department of Mathematics,
Waseda University, Ohkubo, Shinjuku-ku,
Tokyo 169-8555, Japan
}
\email{noboru@moegi.waseda.jp}
\date{\today}
\begin{abstract}
Homotopy classes of nanowords and nanophrases are combinatorial
 generalizations of virtual knots and links.
Goussarov, Polyak and Viro defined finite type invariants for virtual
 knots and links via semi-virtual crossings.
We extend their definition to nanowords and nanophrases.
We study finite type invariants of low degrees.
In particular, we show that the linking matrix and $T$ invariant defined
 by Fukunaga are finite type of degree one and degree two respectively.
We also give a finite type invariant of degree $4$ for open homotopy of
 Gauss words.
\end{abstract}
\thanks{The first author is  supported by a Scholarship from the
Ministry of Education, Culture, Sports, Science and Technology of
Japan. 
The second author was a Research Fellow of the Japan Society for the
Promotion of Science.
This work was partly supported by KAKENHI} 
\maketitle
\section{Introduction}
A Gauss word is a word such that any letter appearing in the word does
so exactly twice and an $r$-component Gauss phrase is a sequence of $r$
finite length words such that their concatenation gives a Gauss word.
Let $\alpha$ be a finite set.
Then an $\alpha$-alphabet is a set which has a map from the set
to $\alpha$.
An $r$-component nanophrase over $\alpha$ is a pair $(\mathcal{A},p)$
where $p$ is an $r$-component Gauss phrase and $\mathcal{A}$ is an
$\alpha$-alphabet consisting of the letters appearing in $p$.
We write $P(\alpha)$ for the set of nanophrases over $\alpha$.
Nanowords are $1$-component nanophrases.
Nanowords and nanophrases were defined by Turaev in \cite{Turaev:Words}
and \cite{Turaev:KnotsAndWords}.
\par
Let $\tau$ be an involution on $\alpha$ and let $S$ be a subset of
$\alpha \times \alpha \times \alpha$.
Using these data, moves are defined on nanophrases.
The moves generate an equivalence relation on nanophrases over $\alpha$
called homotopy.
Different choices of $\alpha$, $\tau$ and $S$ may give a different
equivalence relation.
One such choice of $\alpha$, $\tau$ and $S$ gives a homotopy for which
the equivalence classes of nanowords over $\alpha$ correspond
bijectively to open virtual knots \cite{Turaev:KnotsAndWords}. 
\par
Finite type invariants for classical knots and links were defined by
Vassiliev in \cite{Vassiliev:1990}.
They can be defined in terms of the crossing change operation.
Finite type invariants for virtual knots and links were defined in the
same way by Kauffman \cite{Kauffman:VirtualKnotTheory}.
Goussarov, Polyak and Viro defined finite type invariants for virtual
knots and links in a different way by introducing a new kind of crossing
called a semi-virtual crossing based on the virtualization operation
(changing a real crossing into a virtual crossing)
\cite{Goussarov/Polyak/Viro:FiniteTypeInvariants}.
Finite type invariants in the sense of Goussarov, Polyak and Viro are
finite type invariants in the sense of Kauffman, but the reverse does
not hold.
\par
In this paper we extend the approach of Goussarov, Polyak and Viro to
define finite type invariants of nanowords and nanophrases.
We extend the definition of nanophrases to allow some letters to be
marked with a dot.
We call such letters semi-letters.
We view nanophrases with semi-letters as elements in
$\Z P(\alpha)$ as follows.
Let $p_A$ be a nanophrase that contains the letter $A$, let $p$ be
the nanophrase derived from $p$ by removing the letter $A$ and let
$p_{\dot{A}}$ be the nanophrase derived from $p_A$ by marking $A$ as
a semi-letter.
Then we define $p_{\dot{A}}$ to be
\begin{equation*}
p_{\dot{A}} = p_A - p.
\end{equation*}
\par
We fix $\alpha$, $\tau$ and $S$ and thus fix a homotopy.
Let $v$ be a homotopy invariant taking values in an additive abelian
group $G$. 
We extend $v$ to $\Z P(\alpha)$ linearly.
Then $v$ is a finite type invariant if there exists an $n$ such that for
all nanophrases with more than $n$ semi-letters, $v(p)$ is $0$.
\par
Finite type invariants of degree $0$ are trivial.
We show that the linking matrix
invariant defined by Fukunaga \cite{Fukunaga:nanophrases2} is a finite
type invariant of degree $1$ (Theorem~\ref{thm:linkingmatrix-deg1}).
Any other finite type invariant of degree $1$ can be calculated from the
linking matrix (Theorem~\ref{thm:deg1-lm}).
Fukunaga's $T$ invariant is a homotopy invariant when $S$ is diagonal
(that is, $S$ has the form 
$\{(a,a,a) \; | \; a \in \alpha\}$) \cite{Fukunaga:nanophrases2}.
In Theorem~\ref{thm:t-degree2} we show that Fukunaga's $T$ invariant is
a finite type invariant of degree $2$.
However, under the same condition on $S$, there exist finite type
invariants of degree $2$ which are independent of $T$
(Theorem~\ref{thm:deg2-tindep}).
\par
Let $v$ be a finite type invariant of degree $n$ for $r$-component
nanophrases over $\alpha$.
The invariant $v$ is a universal invariant of degree $n$ if for any
other finite type invariant $v^\prime$ of degree $n$, there exists a
homomorphism $f$ such that $v^\prime$ is equal to $f \circ v$.
Goussarov, Polyak and Viro defined universal invariants for virtual
knots and links.
Following their approach we define universal invariants for
nanophrases.
Up to isomorphism, the image of $\Z P(\alpha)$ under a universal
invariant of degree $n$ does not depend on the universal invariant.
We define $G_n(\alpha,\tau,S,r)$ to be this image.
In Theorem~\ref{thm:deg2-nanoword} we calculate $G_2(\alpha,\tau,S,1)$
for all $\alpha$ and $\tau$ and for certain $S$.
\par
In Section~\ref{sec:gausswords} we consider nanowords in the case where
$\alpha$ is a single element, $\tau$ is the identity map and $S$ is
diagonal.
In this case, the map to $\alpha$ can be forgotten and nanowords are
just Gauss words.
Thus the homotopy given by $\alpha$, $\tau$ and $S$ is called homotopy
of Gauss words (it was called open homotopy of Gauss words in
\cite{Gibson:gauss-word}).
It was shown independently in \cite{Gibson:gauss-word} and
\cite{Manturov:freeknots} that Gauss word homotopy that there exists
Gauss words which are not homotopically equivalent to the trivial Gauss
word disproving a conjecture by Turaev \cite{Turaev:Words}.
For homotopy of Gauss words, we show that although there are no finite
type invariants of degree $1$, $2$ or $3$, there is a unique finite type
invariant of degree $4$ which takes values in $\cyclic{2}$
(Theorem~\ref{thm:gw}).
This invariant is easy to calculate and gives another way to show the
existence of homotopically non-trivial Gauss words.
\section{Nanowords and nanophrases}
In this section we recall the definitions of nanowords, nanophrases and
their homotopies.
All definitions in this section were originally given by Turaev in
\cite{Turaev:Words} and \cite{Turaev:KnotsAndWords}.
\par
A \emph{word} of length $m$ is a finite sequence of $m$ letters.
The unique word of length $0$ is called the \emph{trivial word} and is
written $\trivial$.
An $r$-component \emph{phrase} is a finite sequence of $r$ words which we
call \emph{components}.
When writing phrases we use the `$|$' symbol to separate components.
For example $ABC|\trivial|B$ is a $3$-component phrase.
The unique $r$-component phrase for which every component is the trivial
word is called the \emph{trivial $r$-component phrase} and is denoted
$\trivialp{r}$.
\par
A \emph{Gauss word} is a word in which any letter appears either exactly
twice or not at all.
Similarly, a \emph{Gauss phrase} is a phrase which satisfies the same
condition.
Alternatively, a phrase is a Gauss phrase if the concatenation of all of
its components forms a Gauss word.
\par
The \emph{rank} of a Gauss word or Gauss phrase is the number of
distinct letters that appears in it.
Note that the rank of a Gauss word must be half its length.
For example, the rank of $ABACBC$ is $3$ and the rank of
$ABA|\trivial|B$ is $2$.
\par
Let $\alpha$ be a finite set.
An \emph{$\alpha$-alphabet} is a set with a map to
$\alpha$.
For a letter $A$ in an $\alpha$-alphabet, its image under the map is
denoted $\xType{A}$.
An $r$-component \emph{nanophrase} over $\alpha$ is a pair
$(\mathcal{A},p)$ where $p$ 
is an $r$-component Gauss phrase and $\mathcal{A}$ is an
$\alpha$-alphabet consisting 
of the letters appearing in $p$.
If a Gauss phrase $p$ only has one component, $p$ is a Gauss word and
$(\mathcal{A},p)$ can be described as a one-component nanophrase over
$\alpha$ or a \emph{nanoword} over $\alpha$.
The rank of a nanophrase is the rank of its Gauss phrase.
\par
Rather than write $(\mathcal{A},p)$ we will often just write $p$ for a
nanophrase.
When writing a nanophrase in this way we do not forget that there is a
map from the set of letters appearing in $p$ to $\alpha$.
\par
When giving specific nanophrases we will sometimes use the notation
$p:x$ where $p$ is a Gauss phrase and $x$ is a word of length $\rank(p)$
in $\alpha$.
Arrange the set of letters appearing in $p$ alphabetically to give a
word $y$ of length $\rank(p)$.
Then the map from the letters in $p$ to $\alpha$ is given as follows.
For each $i$, the $i$th letter in $y$ maps to the $i$th letter in $x$.
\begin{ex}
Let $p$ be the nanophrase $AB|A|B$ where $\xType{A}$ is $a$ and
 $\xType{B}$ is $b$ for some $a$ and $b$ in $\alpha$.
Then we can write $p$ as $AB|A|B:ab$.
\par
On the other hand, let $q$ be $EBC|B|CE:abb$.
Then $q$ is the nanophrase $EBC|B|CE$ where $\xType{B}$ is $a$,
 $\xType{C}$ is $b$ and $\xType{E}$ is $b$.
\end{ex}
Two nanophrases over $\alpha$, $(\mathcal{A},p)$ and $(\mathcal{B},q)$
are \emph{isomorphic} if there exists a bijection from $\mathcal{A}$ to
$\mathcal{B}$ which preserves the map to $\alpha$ and, when applied
letterwise to $p$ gives $q$.
\par
Fixing $\alpha$, let $\tau$ be an involution on $\alpha$ ($\tau$ is a
map from $\alpha$ to $\alpha$ such that $\tau \circ \tau$ is the
identity map) and let $S$ be a subset of 
$\alpha \times \alpha \times \alpha$. 
We say that $S$ is \emph{diagonal} if it has the form
$\{(a,a,a) \; | \; a \in \alpha\}$.
The triple $(\alpha,\tau,S)$ is called \emph{homotopy data}.
\emph{Homotopy moves} for nanophrases are defined as follows
\begin{itemize}
\item[H1:]
$(\mathcal{A},xAAy) \leftrightarrow (\mathcal{A}-\{A\},xy)$
\item[H2:]
$(\mathcal{A},xAByBAz) \leftrightarrow (\mathcal{A}-\{A,B\},xyz)$,
if $\xType{A}=\tau(\xType{B})$
\item[H3:]
$(\mathcal{A},xAByACzBCt) \leftrightarrow (\mathcal{A},xBAyCAzCBt)$,
if $(\xType{A},\xType{B},\xType{C}) \in S$
\end{itemize}
where $x$, $y$, $z$ and $t$ represent arbitrary sequences of letters,
possibly including the `$|$' or `$\trivial$' symbols so that both sides
of each move are nanophrases.
\par
\emph{Homotopy} is the equivalence relation of nanophrases over $\alpha$
generated by isomorphism and the three homotopy moves.
The equivalence relation is dependent on the choice of the homotopy data
$(\alpha, \tau, S)$, so
different choices of homotopy data may give different equivalence
relations.
\begin{rem}\label{rem:vknot-homotopy}
Let $\alpha_{vk}$ be the set $\{\aplus, \aminus, \bplus, \bminus \}$
 and let $\tau_{vk}$ be the involution on $\alpha$ where $\aplus$ maps
 to $\bminus$ and $\aminus$ maps to $\bplus$.
Let $S_{vk}$ be the set
\begin{align*}
S_{vk} = \{ 
& (\aplus,\aplus,\aplus), 
  (\aplus,\aplus,\aminus), 
  (\aplus,\aminus,\aminus), \\
& (\aminus,\aminus,\aminus),
  (\aminus,\aminus,\aplus),
  (\aminus,\aplus,\aplus), \\
& (\bplus,\bplus,\bplus),
  (\bplus,\bplus,\bminus),
  (\bplus,\bminus,\bminus), \\
& (\bminus,\bminus,\bminus),
  (\bminus,\bminus,\bplus),
  (\bminus,\bplus,\bplus)
 \}.
\end{align*}
Turaev showed that the set of homotopy classes of nanowords over
 $\alpha_{vk}$ under the homotopy given by
 $(\alpha_{vk},\tau_{vk},S_{vk})$ is in bijective 
 correspondence with the set of open virtual knots
 \cite{Turaev:KnotsAndWords}.
He also showed that, under the same homotopy, the set of homotopy
 classes of nanophrases over $\alpha_{vk}$ is in bijective
 correspondence with the set of stable equivalence classes of ordered
 pointed link diagrams on oriented surfaces (see Section~6.3 of
 \cite{Turaev:KnotsAndWords} for this result and for definitions of
 ordered pointed link diagrams and stable equivalence).
\par
Virtual knots and links correspond to a nanophrase homotopy where a
 shift move is permitted \cite{Turaev:KnotsAndWords}.
See Section~\ref{sec:closed-homotopy} for further details.
\end{rem}
\section{Finite type invariants}
Let $P(\alpha,\tau,S,r)$ be the set of homotopy classes of $r$-component
nanophrases under the homotopy given by $(\alpha,\tau,S)$.
Let $\Z P(\alpha,\tau,S,r)$ be the free abelian group generated by the
elements of $P(\alpha,\tau,S,r)$.
\par
Let $p$ be a nanophrase in $P(\alpha,\tau,S,r)$ which has the form
$xAyAz$ for some letter $A$,  where $x$, $y$ and $z$ are arbitrary
sequences of letters, possibly including the `$|$' or `$\trivial$'
symbols.
We write $x\dot{A}y\dot{A}z$ to denote the formal sum given by
\begin{equation}\label{eqn:singular-letter}
x\dot{A}y\dot{A}z = xAyAz - xyz.
\end{equation}
Here $\dot{A}$ is called a \emph{semi-letter} and 
$\xType{\dot{A}}$ is equal to $\xType{A}$.
By applying Equation~\eqref{eqn:singular-letter} recursively,
nanophrases may contain an arbitrary number of semi-letters.
Note that in \cite{Fujiwara:ftype-arnold}, Fujiwara made the same
definition for semi-letters of nanowords representing plane curves,
although he called them singular letters.
\par
Let $v$ be a homotopy invariant for nanophrases in
$P(\alpha,\tau,S,r)$ which takes values in an additive abelian group.
Then we can extend $v$ to $\Z P(\alpha,\tau,S,r)$ by addition.
In particular, for nanophrases with semi-letters, we have
\begin{equation*}
v(x\dot{A}y\dot{A}z) = v(xAyAz) - v(xyz).
\end{equation*}
\par
We say that $v$ is a \emph{finite type invariant} if there exists an
integer $n$ such that for any nanophrase $p$ with more than $n$
semi-letters, $v(p)$ is $0$. 
The least such $n$ is called the \emph{degree} of $v$.
\par
In Remark~\ref{rem:vknot-homotopy} we noted that open virtual knots
correspond to homotopy classes given by particular homotopy data.
In this case our semi-letters correspond to Goussarov, Polyak and
Viro's semi-virtual crossings
\cite{Goussarov/Polyak/Viro:FiniteTypeInvariants}.
\begin{rem}
Using nanowords, the second author gave a systematic construction for a
 large family of finite type invariants of plane curves in
 \cite{Ito:Ftype-curves}.
Fujiwara also studied finite type invariants of plane curves using
 nanowords \cite{Fujiwara:ftype-arnold}.
However, in both cases, invariance under homotopy moves was not
 considered.
\end{rem}
For any homotopy, finite type invariants of degree $0$ are trivial in
the following sense.
\begin{prop}
Let $u$ be a finite type invariant of degree $0$ for some homotopy.
Then for any $r$-component nanophrase $p$, $u(p)$ is equal to
 $u(\trivialp{r})$.
\end{prop}
\begin{proof}
We prove by induction on the rank of $p$.
If $p$ has rank $0$, then $p$ is $\trivialp{r}$ and so the result is
 true.
Now suppose that $p$ has rank $n$ (greater than $0$) and that the result
 is true for any nanophrase of rank less than $n$.
So $p$ contains a letter, say $A$, and $p$ can be written in the form
 $xAyAz$.
Then we have
\begin{equation*}
u(x\dot{A}y\dot{A}z) = u(xAyAz) - u(xyz)
\end{equation*}
which implies
\begin{equation*}
u(xAyAz) = u(xyz)
\end{equation*}
because $u$ is a finite type invariant of degree $0$.
Since $xyz$ is a nanophrase with rank less than $n$, $u(xyz)$ equals
 $u(\trivialp{r})$ by the induction assumption.
Thus $u(p)$ is equal to $u(\trivialp{r})$.
\end{proof}
\section{Angle bracket formulae}
Let $p$ and $q$ be $r$-component nanophrases.
The nanophrase $q$ is a \emph{subphrase} of $p$, written 
$q \subphrase p$, if it can be obtained by
deleting letters from $p$.
By definition $p$ is a subphrase of itself.
Note that the trivial $r$-component nanophrase
$\trivialp{r}$ is a subphrase of any $r$-component nanophrase. 
If $p$ has rank $n$, $p$ has exactly $2^n$ subphrases.
\begin{ex}
Let $p$ be the nanophrase $ABC|BA|C:abc$ for some $a$, $b$ and $c$ in
 $\alpha$.
Then $p$ has $8$ subphrases, $p$ itself, $\trivialp{3}$ and the
 following $6$ others:
$AB|BA|\trivial:ab$, $AC|A|C:ac$, $BC|B|C:bc$, $A|A|\trivial:a$,
 $B|B|\trivial:b$ and $C|\trivial|C:c$.
\par
Note that the map to $\alpha$ is important.
For example, the nanophrase $A|A|\trivial:b$ is not a subphrase of $p$
 unless $b$ is equal to $a$.
Note also that the number of components of the subphrase and the
 nanophrase should be the same.
The nanophrase $AB|BA:ab$ is not a subphrase of $p$.
\end{ex}
For two $r$-component nanophrases $p$ and $q$ we define the angle
bracket $\langle q, p \rangle$ to be the number of subphrases of $p$
that are isomorphic to $q$.
By definition, for any $r$-component nanophrase $p$, 
$\langle p, p \rangle$ and $\langle \trivialp{r}, p \rangle$ are both
equal to $1$. 
\begin{ex}
Let $p$ be the nanophrase $ABC|BA|C:abc$ from the previous example.
Then $\langle AB|BA|\trivial:ab , p \rangle$ is $1$.
If $a$ is not equal to $b$, $\langle AC|A|C:ac , p \rangle$ is equal to
 $1$.
On the other hand, if $a$ is equal to $b$ then 
$\langle AC|A|C:ac , p \rangle$ is $2$ because $AC|A|C:ac$ and
 $BC|B|C:bc$ are isomorphic.
\end{ex}
We extend the angle bracket linearly in both terms, so that it is a map
from $\Z P(\alpha,\tau,S,r) \times \Z P(\alpha,\tau,S,r)$ to 
$\Z_{\geq 0}$ by
\begin{equation*}
\langle t + u, p \rangle = \langle t, p \rangle + \langle u, p \rangle
\end{equation*}
and 
\begin{equation*}
\langle p, t + u \rangle = \langle p, t \rangle + \langle p, u \rangle,
\end{equation*}
for all elements $t$ and $u$ in $\Z P(\alpha,\tau,S,r)$.
Given an element $u$ in $\Z P(\alpha,\tau,S,r)$ we call the map $f_u$
from $\Z P(\alpha,\tau,S,r)$ to $\Z_{\geq 0}$ given by
\begin{equation*}
f_u(p) = \langle u, p \rangle
\end{equation*}
an \emph{angle bracket formula}.
\begin{rem}
The definition of angle bracket formulae corresponds to the definition
 of Gauss diagram formulae in
 \cite{Goussarov/Polyak/Viro:FiniteTypeInvariants}.
\end{rem}
\begin{lem}\label{lem:singular-plusn}
Let $p$ be an $r$-component nanophrase of rank $n$ and let $q$ be an
 $r$-component nanophrase with more than $n$ semi-letters.
Then $\langle p, q \rangle$ is equal to $0$.
\end{lem}
\begin{proof}
For $r$-component nanophrases $y$ and $z$, define $[y,z]$ by 
\begin{equation*}
[y,z] =
\begin{cases}
1 & \text{if $y$ is a subphrase of $z$} \\
0 & \text{otherwise}.
\end{cases}
\end{equation*}
We extend this notation linearly in the second term to any element in
$\Z P(\alpha,\tau,S,r)$, so 
\begin{equation*}
[y,x_1 + x_2] = [y,x_1] + [y,x_2]
\end{equation*}
for all $x_1$ and $x_2$ in $\Z P(\alpha,\tau,S,r)$.
\par
For an $r$-component nanophrase $u$, let $M(u)$ be the set of subphrases
 of $u$ which are isomorphic to $p$.
Then, by definition, $\langle p, u \rangle$ is equal to the number of
 elements in $M(u)$ and so can be written
\begin{equation*}
\langle p, u \rangle = \sum_{u^\prime \in M(u)}[u^\prime,u].
\end{equation*}
Then for any subphrase $v$ of $u$, $\langle p, v \rangle$ can be written
\begin{equation}\label{eqn:squarebracket}
\langle p, v \rangle = \sum_{u^\prime \in M(u)}[u^\prime,v].
\end{equation}
\par
We say that an $r$-component nanophrase $u$ is a \emph{resolution} of
 $q$ (written $u \dashv q$) if $u$ is derived from $q$ by taking each
 semi-letter in $q$ and either removing the letter or by removing the
 dot from the letter.
If $q$ has $m$ semi-letters, there are $2^m$ different resolutions
 of $q$.
Let $q^\prime$ be the nanophrase derived from $q$ by removing the dots
 from all the semi-letters.
Then $q^\prime$ is a resolution of $q$ and all other resolutions of $q$
 are subphrases of $q^\prime$. 
Let $u$ be a resolution of $q$.
Then we define $\delta(u,q)$ by
\begin{equation*}
\delta(u,q) = \rank(q) - \rank(u).
\end{equation*} 
\par
Using resolutions, $\langle p, q \rangle$ can be written as
\begin{equation*}
\langle p, q \rangle = \sum_{u \dashv q} (-1)^{\delta(u,q)}\langle p,u \rangle.
\end{equation*}
Using Equation~\eqref{eqn:squarebracket}, this becomes
\begin{align*}
\langle p, q \rangle & = \sum_{u \dashv q} (-1)^{\delta(u,q)} \sum_{v \in
 M(q^\prime)}[v,u] \\ 
& = \sum_{v \in M(q^\prime)} \sum_{u \dashv q} (-1)^{\delta(u,q)}
 [v,u] \\
& = \sum_{v \in M(q^\prime)} [v,q].
\end{align*}
\par
For any $v$ in $M(q^\prime)$, the rank of $v$ is $n$.
Since $q$ is a nanophrase with more than $n$ semi-letters, there
 must exist a semi-letter in $q$ which does not appear (as a
 letter without a dot) in $v$.
Call this letter $\dot{A}$.
Let $q_A$ be the nanophrase derived from $q$ by changing $\dot{A}$ to
 $A$.
Let $q_0$ be the nanophrase derived from $q$ by removing $\dot{A}$.
Then
\begin{equation*}
[v,q] = [v,q_A] - [v,q_0] = 0
\end{equation*}
because $[v,q_A]$ is equal to $[v,q_0]$.
\par
Thus
\begin{equation*}
\langle p, q \rangle = \sum_{v \in M(q^\prime)} [v,q] = 0
\end{equation*}
and the proof is complete.
\end{proof}
The \emph{degree} of an element of $\Z P(\alpha,\tau,S,r)$ is the
maximum of the ranks of the terms.
The \emph{degree} of an angle bracket formula $\langle u, p \rangle$ is
the degree of $u$.
\begin{prop}\label{prop:angleisft}
Suppose $v$ is a homotopy invariant which is given by an angle
 bracket formula with degree $m$.
Then $v$ is a finite type invariant of degree less than or equal to
 $m$.
\end{prop}
\begin{proof}
Suppose $v(p)$ is given by the angle bracket formula 
 $\langle u, p \rangle$ for some $u$ in $\Z P(\alpha,\tau,S,r)$.
Then we just need to show that for any nanophrase $q$ with more than $m$
 semi-letters, $\langle u, p \rangle$ is equal to $0$.
This follows from Lemma~\ref{lem:singular-plusn}.
\end{proof}
\begin{rem}
If $v$ is a homotopy invariant which is given by an angle bracket
 formula with degree $m$, it is possible that $v$ is a finite type
 invariant of degree strictly less than $m$.
Later, in Example~\ref{ex:angle-ft-degree-a} and
 Example~\ref{ex:angle-ft-degree-b}, we give some examples of finite
 type invariants of degree $1$ which are defined by angle bracket
 formulae of degree $2$. 
\end{rem}
\section{Universal invariants}\label{sec:universal}
Let $v$ be a finite type invariant for $P(\alpha,\tau,S,r)$ taking
values in an abelian group $G$.
We say that $v$ is a \emph{universal invariant} of degree $n$ if $v$ has
degree less than or equal to $n$ and for
every finite type invariant $v^\prime$ of degree less than or equal to
$n$ taking values in some abelian group $H$, there exists a homomorphism
$f$ from $G$ to $H$ such that the following diagram
\begin{equation*}
\begin{minipage}[c]{3cm}
\xymatrix{
\Z P(\alpha,\tau,S,r) \ar[r]^-v \ar[rd]_{v^\prime} &
G \ar[d]^{f} \\
&H
}
\end{minipage}
\end{equation*}
is commutative.
In other words, if $p$ and $q$ are two $r$-component nanophrases over
$\alpha$ which can be distinguished by a finite type invariant of degree
$n$ and $v$ is a universal invariant of degree $n$, then $v(p)$ is not
equal to $v(q)$.
\par
In \cite{Goussarov/Polyak/Viro:FiniteTypeInvariants}, Goussarov, Polyak
and Viro defined a universal invariant for finite type
invariants of virtual knots and links.
In a similar way, we now define universal invariants for homotopies of
nanophrases.
\par
Let $\Z \mathcal{I}_r(\alpha)$ be the additive abelian group generated
by $r$-component nanophrases modulo isomorphism.
Then $\Z P(\alpha,\tau,S,r)$ is $\Z \mathcal{I}_r(\alpha)$ modulo the
three homotopy moves.
\par 
Let $G(\alpha,\tau,S,r)$ be the group given by $\Z \mathcal{I}_r(\alpha)$
modulo the following three types of relations.
The first type of relation has the form
\begin{equation*}
xAAy = 0,
\end{equation*}
where $x$ and $y$ are arbitrary sequences of letters possibly including
the `$|$' symbol so that $xy$ is a nanophrase.
The second type of relation has the form
\begin{equation*}
xAByBAz + xAyAz + xByBz = 0,
\end{equation*}
where $x$, $y$ and $z$ are arbitrary sequences of letters possibly including
the `$|$' symbol so that $xyz$ is a nanophrase and $\xType{A}$ is equal
to $\tau(\xType{B})$.
The third type of relation has the form
\begin{align*}
xAByACzBCt + xAByAzBt + xAyACzCt + xByCzBCt \\ =
xBAyCAzCBt + xBAyAzBt + xAyCAzCt + xByCzCBt,
\end{align*}
where $x$, $y$, $z$ and $t$ are arbitrary sequences of letters possibly
including the `$|$' symbol so that $xyzt$ is a nanophrase and the triple
$(\xType{A},\xType{B},\xType{C})$ is in $S$.
The relations hold for any set of nanophrases matching the terms.
\begin{rem}
When $(\alpha,\tau,S)$ is $(\alpha_{vk},\tau_{vk},S_{vk})$ (see
 Remark~\ref{rem:vknot-homotopy}), these relations are equivalent to
 those appearing in Section~2.5 of
 \cite{Goussarov/Polyak/Viro:FiniteTypeInvariants}.
\end{rem}
We define a map $\theta_r$ from $\Z \mathcal{I}_r(\alpha)$ to itself as
follows.
For an $r$-component nanophrase $p$, $\theta_r(p)$ is the sum of all the
subphrases of $p$ considered as an element of
$\Z \mathcal{I}_r(\alpha)$.
We then extend this definition linearly to all of 
$\Z \mathcal{I}_r(\alpha)$.
Note that for a nanophrase $p$, $\theta_r(p)$ can be written as
\begin{equation*}
\theta_r(p) = \sum_{q \subphrase p} q.
\end{equation*}
\begin{ex}
Consider the nanophrase $AB|AB:aa$ for some $a$ in $\alpha$.
Then $\theta_r(AB|AB:aa)$ is given by
\begin{equation*}
\theta_r(AB|AB:aa) = AB|AB:aa + A|A:a + B|B:a + \trivial|\trivial.
\end{equation*}
The nanophrases $A|A:a$ and $B|B:a$ are isomorphic and so they are
 equivalent in $\Z \mathcal{I}_r(\alpha)$.
Thus $\theta_2(AB|AB)$ can be given more simply by
\begin{equation*}
\theta_2(AB|AB:aa) = AB|AB:aa + 2 A|A:a + \trivial|\trivial.
\end{equation*}
\end{ex} 
\begin{ex}
For some $a$ in $\alpha$, $\theta_2(A|BAB:aa - AA|BB:aa)$ is
\begin{equation*}
A|BAB:aa + A|A:a - AA|BB:aa - AA|\trivial:a.
\end{equation*}
\end{ex}
\begin{rem}
The map $\theta_r$ corresponds to the map $I$ in
 \cite{Goussarov/Polyak/Viro:FiniteTypeInvariants}.
\end{rem}
We define another map $\phi_r$ from $\Z \mathcal{I}_r(\alpha)$ to
itself as follows.
For an $r$-component nanophrase $p$,
\begin{equation*}
\phi_r(p) = \sum_{q \subphrase p} (-1)^{\rank(p)-\rank(q)} q.
\end{equation*}
We then extend this linearly to all of $\Z \mathcal{I}_r(\alpha)$.
\begin{ex}
Consider the nanophrase $AB|AB:aa$.
Then 
\begin{equation*}
\phi_2(AB|AB:aa) = AB|AB:aa - 2 A|A:a + \trivial|\trivial.
\end{equation*}
\end{ex}
\begin{prop}
The map $\theta_r$ is a bijection.
Its inverse is given by $\phi_r$.
\end{prop}
\begin{proof}
We will show that for any $r$-component nanophrase $p$,
$\phi_r \circ \theta_r(p)$ and $\theta_r \circ \phi_r(p)$ are both equal
 to $p$.
By extending this linearly to all of $\Z \mathcal{I}_r(\alpha)$, this
 implies that $\phi_r \circ \theta_r$ and $\theta_r \circ \phi_r$ are
 both equivalent to the identity map and this gives the required
 result.
\par
Now
\begin{align*}
\phi_r \circ \theta_r(p) & = \phi_r(\sum_{q \subphrase p} q) \\
& = \sum_{q \subphrase p} \phi_r(q) \\
& = \sum_{q \subphrase p} \sum_{s \subphrase q} (-1)^{\rank(q)-\rank(s)} s.
\end{align*}
Rearranging the terms and write $I(s,p)$ for the set of nanophrases $q$
 satisfying $s \subphrase q \subphrase p$,
this becomes
\begin{equation*}
\phi_r \circ \theta_r(p) = 
\sum_{s \subphrase p} \left(
\sum_{q \in I(s,p)}
(-1)^{\rank(q)-\rank(s)}
\right) s.
\end{equation*}
We write $c(s)$ for
\begin{equation}\label{eqn:cs}
\sum_{q \in I(s,p)}
(-1)^{\rank(q)-\rank(s)}.
\end{equation}
Then
\begin{equation*}
c(p) = \sum_{q \in I(p,p)}
(-1)^{\rank(q)-\rank(p)} = 1.
\end{equation*}
\par
Let $s$ be any subphrase of $p$ other than $p$ itself.
Then there is a letter, say $A$, which appears in $p$ but not in $s$.
Consider the set of subphrases $q$ in the sum for $c(s)$.
Exactly half of these subphrases contain $A$ and the other half do not.
If $q$ is a subphrase containing $A$, let $q_A$ be the subphrase derived
 from $q$ by deleting $A$.
Note that if $q$ appears in the sum for $c(s)$, so does $q_A$.
Also note that the map defined by deleting the letter $A$ gives a
 bijection from the set of subphrases in the sum for $c(s)$ containing
 the letter $A$ to the set of subphrases in the sum for $c(s)$ which do
 not contain $A$.
Write $\mathcal{Q}(s)$ for the set of subphrases in the sum for $c(s)$
 containing the letter $A$.
Then Equation~\eqref{eqn:cs} becomes
\begin{equation*}
c(s) = \sum_{q \in \mathcal{Q}(s)}
\left(
(-1)^{\rank(q)-\rank(s)} + (-1)^{\rank(q_A)-\rank(s)}
\right).
\end{equation*}
However, the rank of $q_A$ is one less than the rank of $q$ which means that
\begin{equation*}
(-1)^{\rank(q)-\rank(s)} + (-1)^{\rank(q_A)-\rank(s)} = 0
\end{equation*}
and so $c(s)$ is equal to $0$ for all subwords $s$ of $p$ except $p$
 itself.
Thus $\phi_r \circ \theta_r(p)$ equals $p$.
\par
On the other hand,
\begin{align*}
\theta_r \circ \phi_r(p) 
& = \theta_r(\sum_{q \subphrase p} (-1)^{\rank(p)-\rank(q)} q) \\
& = \sum_{q \subphrase p} (-1)^{\rank(p)-\rank(q)} \theta_r(q) \\
& = \sum_{q \subphrase p} (-1)^{\rank(p)-\rank(q)} \sum_{s \subphrase q} s.
\end{align*}
Rearranging the terms, this becomes
\begin{equation*}
\theta_r \circ \phi_r(p) = 
\sum_{s \subphrase p} \left(
\sum_{q \in I(s,p)}
(-1)^{\rank(p)-\rank(q)}
\right) s.
\end{equation*}
It is then easy to check that this is equal to $p$ using a similar
 method to the one we used for $\phi_r \circ \theta_r$.
\end{proof}
\begin{prop}\label{prop:thetar-isomorphism}
The map $\theta_r$ induces an isomorphism from $\Z P(\alpha,\tau,S,r)$ to
 $G(\alpha,\tau,S,r)$.
\end{prop}
\begin{proof}
We show that $\theta_r$ is a homomorphism from $\Z P(\alpha,\tau,S,r)$ to
 $G(\alpha,\tau,S,r)$ and that $\phi_r$ is a homomorphism from
 $G(\alpha,\tau,S,r)$ to $\Z P(\alpha,\tau,S,r)$.
\par
We start with $\theta_r$ and check the relations given by each homotopy
 move.
For the move H1 we need to show that
\begin{equation}\label{eqn:hom-h1-toshow}
\theta_r(xAAy) - \theta_r(xy) = 0
\end{equation}
for all nanophrases $xAAy$.
Now 
\begin{equation}\label{eqn:hom-h1}
\theta_r(xAAy) - \theta_r(xy) = 
\sum_{q \subphrase xAAy} q - \sum_{q \subphrase xy} q.
\end{equation}
Note that the set of subphrases of $xAAy$ which do not contain $A$
 are exactly the set of subphrases of $xy$.
Writing $\mathcal{Q}(p)$ for the set of subphrases of $xAAy$ which
 contain $A$, Equation~\eqref{eqn:hom-h1} becomes
\begin{equation*}
\theta_r(xAAy) - \theta_r(xy) = 
\sum_{q \in \mathcal{Q}(p)} q
 + \sum_{q \subphrase xy} q
 - \sum_{q \subphrase xy} q.
\end{equation*}
However, in $G(\alpha,\tau,S,r)$ any nanophrase of the form $uAAv$ is
 $0$ by a relation of the first type.
Thus every nanophrase $q$ in $\mathcal{Q}(p)$ is $0$ and so
 Equation~\eqref{eqn:hom-h1-toshow} holds.
\par
For the move H2 we need to show that
\begin{equation}\label{eqn:hom-h2-toshow}
\theta_r(xAByBAz) - \theta_r(xyz) = 0
\end{equation}
for all nanophrases $xAByBAz$ where $\xType{A}$ is equal to
 $\tau(\xType{B})$.
Write $\mathcal{Q}_{AB}(p)$ for the set of subphrases of $xAByBAz$ which
 contain both $A$ and $B$, $\mathcal{Q}_{A}(p)$ for the set of
 subphrases of $xAByBAz$ which contain $A$ and not $B$,
 $\mathcal{Q}_{B}(p)$ for the set of subphrases of $xAByBAz$ which
 contain $B$ and not $A$ and $\mathcal{Q}(p)$ for the set of subphrases
 of $xAByBAz$ which do not contain $A$ or $B$.
Note that $\mathcal{Q}(p)$ is also the set of subphrases of $xyz$.
Then
\begin{multline}\label{eqn:hom-h2}
\theta_r(xAByBAz) - \theta_r(xyz) = \\
\sum_{q \in \mathcal{Q}_{AB}(p)} q
+ \sum_{q \in \mathcal{Q}_{A}(p)} q
+ \sum_{q \in \mathcal{Q}_{B}(p)} q
+ \sum_{q \in \mathcal{Q}(p)} q
- \sum_{q \in \mathcal{Q}(p)} q.
\end{multline}
For a subphrase $q_{AB}$ in $\mathcal{Q}_{AB}(p)$, let $q_A$ be the
 nanophrase derived from $q_{AB}$ by deleting the letter $B$ and let
 $q_B$ be the nanophrase derived from $q_{AB}$ by deleting the letter
 $A$.
Then the map taking $q_{AB}$ to $q_A$ gives a bijection from
 $\mathcal{Q}_{AB}(p)$ to $\mathcal{Q}_{A}(p)$ and the map taking
 $q_{AB}$ to $q_B$ gives a bijection from $\mathcal{Q}_{AB}(p)$ to
 $\mathcal{Q}_{B}(p)$.
Then Equation~\eqref{eqn:hom-h2} can be rewritten as
\begin{equation*}
\theta_r(xAByBAz) - \theta_r(xyz) = 
\sum_{q_{AB} \in \mathcal{Q}_{AB}(p)} (q_{AB} + q_A + q_B).
\end{equation*}
Now since $\xType{A}$ is equal to $\tau(\xType{B})$ there is a relation
 of the second type which gives
\begin{equation*}
q_{AB} + q_A + q_B = 0
\end{equation*}
for each $q_{AB}$ in $\mathcal{Q}_{AB}$,
and so Equation~\eqref{eqn:hom-h2-toshow} holds.
\par
For the move H3 we need to show that
\begin{equation}\label{eqn:hom-h3-toshow}
\theta_r(xAByACzBCt) - \theta_r(xBAyCAzCBt) = 0
\end{equation}
for all nanophrases $xAByACzBCt$ where $(\xType{A},\xType{B},\xType{C})$
 is in $S$.
Let $\mathcal{Q}_{ABC}(p)$ be the set of subphrases of $xAByACzBCt$
 which contain the letters $A$, $B$ and $C$.
For a nanophrase $q_{ABC}$ in $\mathcal{Q}_{ABC}(p)$, we can derive
 seven more nanophrases by deleting different subsets of the letters
 $A$, $B$ and $C$.
The resulting nanophrases are written $q_{AB}$, $q_{AC}$, $q_{BC}$, 
 $q_A$, $q_B$, $q_C$ and $q$, using similar notation to that used in the
 H2 case.
\par
By applying the H3 move to the letters $A$, $B$ and $C$ to a nanophrase
 $q_{ABC}$ in $\mathcal{Q}_{ABC}(p)$ we get a new nanophrase which we
 label $q^{\prime}_{ABC}$.
As for $q_{ABC}$, we derive seven nanophrases from $q^{\prime}_{ABC}$ by
 deleting  different subsets of the letters $A$, $B$ and $C$.
The resulting nanophrases are written $q^{\prime}_{AB}$,
 $q^{\prime}_{AC}$, $q^{\prime}_{BC}$,  
 $q^{\prime}_A$, $q^{\prime}_B$, $q^{\prime}_C$ and $q^{\prime}$.
Then we have
\begin{multline*}
\theta_r(xAByACzBCt) = \\
\sum_{q_{ABC} \in \mathcal{Q}_{ABC}(p)}
\left( q_{ABC}
 + q_{AB} + q_{AC} + q_{BC}
 + q_A + q_B + q_C
 + q \right)
\end{multline*}
and
\begin{multline*}
\theta_r(xBAyCAzCBt) = \\
\sum_{q_{ABC} \in \mathcal{Q}_{ABC}(p)}
\left( q^{\prime}_{ABC}
 + q^{\prime}_{AB} + q^{\prime}_{AC} + q^{\prime}_{BC}
 + q^{\prime}_A + q^{\prime}_B + q^{\prime}_C
 + q^{\prime}\right).
\end{multline*}
Note that $q^{\prime}_A$, $q^{\prime}_B$, $q^{\prime}_C$ and $q^{\prime}$
 are equal to $q_A$, $q_B$, $q_C$ and $q$ respectively.
Thus
\begin{multline*}
\theta_r(xAByACzBCt) - \theta_r(xBAyCAzCBt)
= \\ \sum_{q_{ABC} \in \mathcal{Q}_{ABC}(p)}
\left(
q_{ABC} + q_{AB} + q_{AC} + q_{BC}
- q^{\prime}_{ABC} - q^{\prime}_{AB} - q^{\prime}_{AC} - q^{\prime}_{BC}
\right).
\end{multline*}
Now since $(\xType{A},\xType{B},\xType{C})$ is in $S$, there is a
 relation of the third type which gives
\begin{equation*}
q_{ABC} + q_{AB} + q_{AC} + q_{BC}
- q^{\prime}_{ABC} - q^{\prime}_{AB} - q^{\prime}_{AC} - q^{\prime}_{BC}
 = 0
\end{equation*}
for each $q_{ABC}$ in $\mathcal{Q}_{ABC}$,
and so Equation~\eqref{eqn:hom-h3-toshow} holds.
\par
We now consider $\phi_r$.
We check each type of relation in $G(\alpha,\tau,S,r)$.
\par
For the first relation we need to show that
\begin{equation}\label{eqn:hom-rel1-toshow}
\phi_r(xAAy) = 0
\end{equation}
for all nanophrases $xAAy$.
Let $\mathcal{Q}_A$ be the set of subphrases of $xAAy$ which contain the
 letter $A$.
For each nanophrase $q_A$ in $\mathcal{Q}_A$ we derive a nanophrase $q$
 be deleting the letter $A$.
Then we have
\begin{equation*}
\phi_r(xAAy) = \sum_{q_{A} \in \mathcal{Q}_{A}(p)}
(-1)^{\rank(p)-\rank(q_{A})}
\left(
q_{A} - q
\right).
\end{equation*}
However, in $P(\alpha,\tau,S,r)$, $q_{A}$ is equal to $q$ for each
 $q_{A}$ in $\mathcal{Q}_A$.
Thus Equation~\eqref{eqn:hom-rel1-toshow} holds.
\par
For the second relation we need to show that
\begin{equation}\label{eqn:hom-rel2-toshow}
\phi_r(xAByBAz) + \phi_r(xAyAz) + \phi_r(xByBz) = 0
\end{equation}
for all nanophrases $xAByBAz$ with $\xType{A}$ equal to
 $\tau(\xType{B})$.
\par
Let $\mathcal{Q}_{AB}(p)$ be the set of subphrases of $xAByBAz$ which
 contain both $A$ and $B$.
For a subphrase $q_{AB}$ in $\mathcal{Q}_{AB}(p)$, let $q_A$ be the
 nanophrase derived from $q_{AB}$ by deleting the letter $A$, let
 $q_B$ be the nanophrase derived from $q_{AB}$ by deleting the letter
 $B$ and let $q$ be the nanophrase derived from $q_A$ by deleting the
 letter $A$.
Then 
\begin{multline*}
\phi_r(xAByBAz) + \phi_r(xAyAz) + \phi_r(xByBz) = \\
\sum_{q_{AB} \in \mathcal{Q}_{AB}(p)}
(-1)^{\rank(p)-\rank(q_{AB})}
\left(
q_{AB} - q_A - q_B + q
+ q_A - q
+ q_B - q
\right) = \\
\sum_{q_{AB} \in \mathcal{Q}_{AB}(p)}
(-1)^{\rank(p)-\rank(q_{AB})}
\left(
q_{AB} - q
\right).
\end{multline*}
Now in $P(\alpha,\tau,S,r)$, $q_{AB}$ is equal to $q$ for each
 $q_{AB}$ in $\mathcal{Q}_{AB}$.
Thus Equation~\eqref{eqn:hom-rel2-toshow} holds.
\par
For the third relation we need to show that
\begin{multline}\label{eqn:hom-rel3-toshow}
\phi_r(xAByACzBCt) + \phi_r(xAByAzBt) + \phi_r(xAyACzCt) +
 \phi_r(xByCzBCt) = \\
\phi_r(xBAyCAzCBt) + \phi_r(xBAyAzBt) + \phi_r(xAyCAzCt) +
 \phi_r(xByCzCBt)
\end{multline}
for all nanophrases $xAByACzBCt$ where $(\xType{A},\xType{B},\xType{C})$
 is in $S$.
\par
Let $\mathcal{Q}_{ABC}(p)$ be the set of subphrases of $xAByACzBCt$
 which contain $A$, $B$ and $C$.
For a nanophrase $q_{ABC}$ in $\mathcal{Q}_{ABC}(p)$, we can derive
 seven more nanophrases by deleting different subsets of the letters
 $A$, $B$ and $C$.
Using the same notation we used before, they are written $q_{AB}$,
 $q_{AC}$, $q_{BC}$, $q_A$, $q_B$, $q_C$ and $q$.
By applying the H3 move to the letters $A$, $B$ and $C$ in $q_{ABC}$ we
 get a new nanophrase which, as before, is labelled $q^{\prime}_{ABC}$.
As before, we derive seven nanophrases from $q^{\prime}_{ABC}$ by 
 deleting  different subsets of the letters $A$, $B$ and $C$.
These nanophrases are written $q^{\prime}_{AB}$,
 $q^{\prime}_{AC}$, $q^{\prime}_{BC}$,  
 $q^{\prime}_A$, $q^{\prime}_B$, $q^{\prime}_C$ and $q^{\prime}$.
However, as we noted before,
 $q^{\prime}_A$, $q^{\prime}_B$, $q^{\prime}_C$ and $q^{\prime}$ 
 are equal to $q_A$, $q_B$, $q_C$ and $q$ respectively.
\par
Writing $\delta(q_{ABC})$ for $\rank(p)-\rank(q_{ABC})$, we have
\begin{equation*}
\phi_r(xAByAzBt) = 
\sum_{q_{ABC} \in \mathcal{Q}_{ABC}(p)}
(-1)^{\delta(q_{ABC})}
\left(
q_{AB} - q_A - q_B + q
\right), 
\end{equation*}
\begin{equation*}
\phi_r(xAyACzCt) = 
\sum_{q_{ABC} \in \mathcal{Q}_{ABC}(p)}
(-1)^{\delta(q_{ABC})}
\left(
q_{AC} - q_A - q_C + q
\right),
\end{equation*}
\begin{equation*}
\phi_r(xByCzBCt) = 
\sum_{q_{ABC} \in \mathcal{Q}_{ABC}(p)}
(-1)^{\delta(q_{ABC})}
\left(
q_{BC} - q_B - q_C + q
\right)
\end{equation*}
and
\begin{multline*}
\phi_r(xAByACzBCt) = \\ 
\sum_{q_{ABC} \in \mathcal{Q}_{ABC}(p)}
(-1)^{\delta(q_{ABC})}
\left(
q_{ABC} - q_{AB} - q_{AC} - q_{BC} + q_A + q_B + q_C - q
\right).
\end{multline*}
We also have
\begin{equation*}
\phi_r(xBAyAzBt) = 
\sum_{q_{ABC} \in \mathcal{Q}_{ABC}(p)}
(-1)^{\delta(q_{ABC})}
\left(
q^{\prime}_{AB} - q_A - q_B + q
\right), 
\end{equation*}
\begin{equation*}
\phi_r(xAyCAzCt) = 
\sum_{q_{ABC} \in \mathcal{Q}_{ABC}(p)}
(-1)^{\delta(q_{ABC})}
\left(
q^{\prime}_{AC} - q_A - q_C + q
\right),
\end{equation*}
\begin{equation*}
\phi_r(xByCzCBt) = 
\sum_{q_{ABC} \in \mathcal{Q}_{ABC}(p)}
(-1)^{\delta(q_{ABC})}
\left(
q^{\prime}_{BC} - q_B - q_C + q
\right)
\end{equation*}
and
\begin{multline*}
\phi_r(xBAyCAzCBt) = \\ 
\sum_{q_{ABC} \in \mathcal{Q}_{ABC}(p)}
(-1)^{\delta(q_{ABC})}
\left(
q^{\prime}_{ABC} - q^{\prime}_{AB} - q^{\prime}_{AC} - q^{\prime}_{BC} +
q_A + q_B + q_C - q
\right).
\end{multline*}
Substituting these equations into
\begin{multline*}
\phi_r(xAByACzBCt) + \phi_r(xAByAzBt) + \phi_r(xAyACzCt) +
 \phi_r(xByCzBCt) \\
- \phi_r(xBAyCAzCBt) - \phi_r(xBAyAzBt) - \phi_r(xAyCAzCt) -
 \phi_r(xByCzCBt)
\end{multline*}
 and using the fact that $q_{ABC}$ is equal to $q^{\prime}_{ABC}$ in
 $P(\alpha,\tau,S,r)$ for each $q_{ABC}$ in $\mathcal{Q}_{ABC}$, gives
Equation~\eqref{eqn:hom-rel3-toshow}.
\end{proof}
\begin{rem}
Proposition~\ref{prop:thetar-isomorphism} corresponds to Theorem~2D in
 \cite{Goussarov/Polyak/Viro:FiniteTypeInvariants}.
\end{rem}
\par
We write $\hat{\theta}_r$ for the isomorphism from
$\Z P(\alpha,\tau,S,r)$ to  $G(\alpha,\tau,S,r)$ induced by $\theta_r$.
\par
We introduce a fourth type of relation parameterized by an integer $n$.
The relation is
\begin{equation*}
p = 0
\end{equation*}
where $p$ is any nanophrase which has rank greater than $n$.
Let $G_n(\alpha,\tau,S,r)$ be the group given by $G(\alpha,\tau,S,r)$
modulo all relations of this fourth type with parameter $n$.
Then $G_n(\alpha,\tau,S,r)$ is generated by the set of $r$-component
nanophrases of rank $n$. 
As this set is finite, $G_n(\alpha,\tau,S,r)$ is a finitely generated
abelian group.
\par
For each positive integer $n$ we define a map $O_n$ from
$P(\alpha,\tau,S,r)$ to itself by
\begin{equation*}
O_n(p) =
\begin{cases}
p & \text{ if } \rank(p) \leq n \\
0 &
 \text{ otherwise}
\end{cases}
\end{equation*}
for any $r$-component nanophrase, and then extending linearly to all of
$P(\alpha,\tau,S,r)$.
Clearly $O_n$ induces a homomorphism from $G(\alpha,\tau,S,r)$ to
$G_n(\alpha,\tau,S,r)$ which we also write $O_n$.
\par
Let $\Gamma_{n,r}$ be the composition of $\hat{\theta}_r$ and $O_n$.
Then $\Gamma_{n,r}$ is a homomorphism from $\Z P(\alpha,\tau,S,r)$ to
$G_n(\alpha,\tau,S,r)$.
\par
For a nanophrase $p$, by linearity of $O_n$, we can write
$\Gamma_{n,r}(p)$ as
\begin{equation*}
\Gamma_{n,r}(p) = \sum_{q \subphrase p} O_n(q).
\end{equation*}
We can rewrite this using angle bracket formulae to get
\begin{equation}\label{eqn:gamma-as-angle}
\Gamma_{n,r}(p) = \sum_{q \in P_{r,n}(\alpha)} 
\langle q,p \rangle q,
\end{equation}
where $P_{r,n}(\alpha)$ is the set of $r$-component nanophrases over
$\alpha$ of rank $n$ or less.
\begin{prop}\label{prop:gamma-universal}
The map $\Gamma_{n,r}$ is a universal invariant of degree $n$.
\end{prop}
\begin{proof}
The fact that $\Gamma_{n,r}$ is a finite type invariant of degree
 $n$ follows directly by applying Proposition~\ref{prop:angleisft} to
 Equation~\eqref{eqn:gamma-as-angle}.
\par
Now let $v$ be a finite type invariant of degree less than or equal to
 $n$ for nanophrases in $P(\alpha,\tau,S,r)$, taking values in some
 abelian group $H$.
We need to show that there exists a homomorphism $f$ from
 $G_n(\alpha,\tau,S,r)$ to $H$ such that $f \circ \Gamma_{n,r}$ is equal
 to $v$.
It is enough to show that $\ker(\Gamma_{n,r})$ is a subgroup of 
$\ker(v)$.
\par
Now because, by Proposition~\ref{prop:thetar-isomorphism}, $\hat{\theta}_r$ is
 an isomorphism, $\ker(\Gamma_{n,r})$ is equal to
 $\hat{\theta}_r^{-1}(\ker(O_n))$.
Let $k$ be an element of $\ker(O_n)$.
Then $k$ can be written
\begin{equation*}
k = \sum_{i=1}^j \lambda_i k_i,
\end{equation*} 
for some $j$ where each $k_i$ is a nanophrase of rank greater than $n$
and each $\lambda_i$ is in $\Z$.
For each $i$, let $\dot{k}_i$ be the nanophrase derived from $k_i$ by
 changing every letter to be a semi-letter.
Then
\begin{equation*}
\dot{k}_i = \sum_{k^\prime \subphrase k}
 (-1)^{(\rank(k)-\rank(k^\prime))} k^\prime
= \hat{\theta}_r^{-1}(k_i).
\end{equation*}
Thus $\hat{\theta}_r^{-1}(k)$ can be written as a sum of nanophrases of rank
 greater than $n$ where every letter is a semi-letter.
Then $v(\hat{\theta}_r^{-1}(k))$ is equal to $0$ because $v$ is a finite type
 invariant of degree $n$.
Thus $\ker(\Gamma_{n,r})$ is a subgroup of $\ker(v)$.
\end{proof}
\begin{cor}
For any degree $n$ finite type invariant of degree $v$ taking values in an
 abelian group $G$, there exists a finite set of elements of $G$,
 $\{g_1, \dotsc, g_m\}$, such that $v$ can be written in the form
\begin{equation*}
v(p) = 
\sum_{i = 1}^m \langle q_i,p \rangle g_i
\end{equation*}
where each $q_i$ is an element of $\Z P_{r,n}(\alpha)$.
\end{cor}
\begin{proof}
This easily follows from Proposition~\ref{prop:gamma-universal} and
 Equation~\eqref{eqn:gamma-as-angle}.
\end{proof} 
\begin{rem}
Proposition~\ref{prop:gamma-universal} corresponds to Theorem~2E in
 \cite{Goussarov/Polyak/Viro:FiniteTypeInvariants}.
\end{rem}
Let $u$ be a universal invariant of degree $n$ for $r$-component
nanophrases.
Then 
\begin{equation*}
u(\Z P(\alpha,\tau,S,r)) \cong \Gamma_{n,r}(\Z P(\alpha,\tau,S,r))
\cong G_n(\alpha,\tau,S,r).
\end{equation*}
Thus we can interpret $G_n(\alpha,\tau,S,r)$ as being the group
isomorphic to the image of $\Z P(\alpha,\tau,S,r)$ under any universal
invariant of degree $n$ for $r$-component nanophrases.
\par
Define $H_n(\alpha,\tau,S,r)$ to be the subgroup of
$G_n(\alpha,\tau,S,r)$ generated by all $r$-component nanophrases except
for the trivial nanophrase $\trivialp{r}$.
We have the following proposition. 
\begin{prop}\label{prop:zh-decomp}
For any $\alpha$, $\tau$, $S$, $r$ and $n$, $G_n(\alpha,\tau,S,r)$ can
 be written in the form
\begin{equation*}
G_n(\alpha,\tau,S,r) \cong \Z \oplus H_n(\alpha,\tau,S,r).
\end{equation*}
\end{prop}
\begin{proof}
The trivial nanophrase $\trivialp{r}$ does not appear in any of the
 relations of $G_n(\alpha,\tau,S,r)$.
Thus $\trivialp{r}$ is a free generator in $G_n(\alpha,\tau,S,r)$.
\end{proof} 
\begin{cor}
For any $\alpha$, $\tau$, $S$ and $r$,
\begin{equation*}
G_0(\alpha,\tau,S,r) \cong \Z.
\end{equation*}
\end{cor}
\begin{proof}
Just observe that $H_0(\alpha,\tau,S,r)$ is trivial by relations of the
 fourth type.
\end{proof}
For any $r$-component nanophrase $p$, using the above proposition and
Equation~\eqref{eqn:gamma-as-angle}, $\Gamma_{n,r}(p)$ can be written
in the form
\begin{equation*}
\Gamma_{n,r}(p) = h + \trivialp{r},
\end{equation*}
where $h$ is an element in $H_n(\alpha,\tau,S,r)$.
This shows that the restriction of $\Gamma_{n,r}$ to
$P(\alpha,\tau,S,r)$ is not surjective on $G_n(\alpha,\tau,S,r)$.
\par
We can agree to normalize our finite type invariant so that
$\trivialp{r}$ maps to $0$.
Let $v$ be a universal invariant normalized in this way.
Then image of $\Z P(\alpha,\tau,S,r)$ under $v$ is isomorphic to
$H_n(\alpha,\tau,S,r)$.
Thus $H_n(\alpha,\tau,S,r)$ can be viewed as the maximal space of values
that nanophrases in $P(\alpha,\tau,S,r)$ can take under normalized
finite type invariants of degree $n$ or less.
In particular, if $H_n(\alpha,\tau,S,r)$ is finite,
finite type invariants of degree $n$ or less can only classify
nanophrases in $P(\alpha,\tau,S,r)$ into a finite number of equivalence
classes.
On the other hand, for all $\alpha$, $\tau$, $S$ and $r$,
$P(\alpha,\tau,S,r)$ has an infinite number of elements.
\par
We end this section by noting some relations between the groups
$G_n(\alpha,\tau,S,r)$.
\begin{prop}
Let $(\alpha,\tau,S)$ be homotopy data and $n$ and $r$ be
positive integers.
Let $S^\prime$ be a subset of $S$.
Let $\beta$ be a set such that the number of elements is less than or
 equal to the number of elements in $\alpha$ and let $\tau_{\beta}$ be an
 involution on $\beta$.
Let $f$ be a surjective map from $\alpha$ to $\beta$ which satisfies
 $\tau_{\beta}(f(a)) = f(\tau(a))$ for all $a$ in $\alpha$. 
Then
\begin{itemize}
\item[(i)]
The identity map gives a homomorphism from $G_{n+1}(\alpha,\tau,S,r)$ to
$G_n(\alpha,\tau,S,r)$;
\item[(ii)]
The identity map gives a surjective homomorphism from
 $G_n(\alpha,\tau,S^\prime,r)$ to $G_n(\alpha,\tau,S,r)$;
\item[(iii)]
There is a surjective homomorphism from $G_n(\alpha,\tau,S,r)$ to
 $G_n(\beta,\tau_\beta,f(S),r)$.
\end{itemize}
\end{prop}
\begin{proof}
Statement~(i) follows immediately from the definitions.
Statement~(ii) also follows easily from the definitions because the
 relations of the third type in $G_n(\alpha,\tau,S,r)$ are a subset of
 those in $G_n(\alpha,\tau,S^\prime,r)$.
\par
We now prove Statement~(iii).
The map $f$ induces a map, which we also call $f$, from $P(\alpha)$ to 
$P(\beta)$ as follows.
Let $(\mathcal{A},p)$ be a nanophrase in $P(\alpha,\tau,S)$ and let
 $\varepsilon$ be the map from $\mathcal{A}$ to $\alpha$.
Then define $\mathcal{B}$ to be a $\beta$-alphabet with the same set as
 $\mathcal{A}$ and the map to $\beta$ given by $f \circ \varepsilon$.
Then $f(\mathcal{A},p)$ is defined to be $(\mathcal{B},p)$.
It is then a simple exercise to check that $f$ induces a homomorphism
 from $G_n(\alpha,\tau,S,r)$ to $G_n(\beta,\tau_\beta,f(S),r)$.
Surjectivity follows from the fact that the original map $f$ is
 surjective.
\end{proof}
\section{Finite type invariants of degree $1$}
The linking matrix of a nanophrase was defined by Fukunaga in
\cite{Fukunaga:nanophrases2}.
It is a homotopy invariant under any homotopy of nanophrases
\cite{Fukunaga:nanophrases2}, \cite{Gibson:factor-homotopy}.
We recall the definition here.
\par
First, let $\pi$ be the abelian group generated by elements in
$\alpha$ with the relations $a + \tau(a) = 0$ for all $a$ in $\alpha$.
In \cite{Fukunaga:nanophrases2}, $\pi$ is written multiplicatively, but
here we will write it additively.
For an $r$-component nanophrase $p$, the \emph{linking matrix}
of $p$ is defined as follows.
Let $\mathcal{A}_{ij}(p)$ be the set of letters which have one occurence
in the $i$th component of $p$ and the other occurence in the $j$th
component of $p$.
Let $l_{ii}(p)$ be $0$ and, when $j$ is not equal to $i$, let
$l_{ij}(p)$ be 
\begin{equation*}
l_{ij}(p) = \sum_{X \in \mathcal{A}_{ij}(p)} \xType{X},
\end{equation*}
for $i$ and $j$ positive integers less than or equal to $r$.
Let $L(p)$ be the symmetric $r \times r$ matrix given by the
$l_{ij}(p)$.
\begin{thm}\label{thm:linkingmatrix-deg1}
The linking matrix is a degree $1$ finite type invariant.
\end{thm}
\begin{proof}
Let $p_{\singular{A}\singular{B}}$ be a nanophrase with two
 semi-letters, $\singular{A}$ and $\singular{B}$. 
Let $p_{AB}$ be the nanophrase given by changing both $\singular{A}$ and
 $\singular{B}$ in $p_{\singular{A}\singular{B}}$ to $A$ and $B$.
Let $p$ be the nanophrase given by removing both $\singular{A}$ and
 $\singular{B}$ from $p_{\singular{A}\singular{B}}$.
The nanophrase $p_A$ is the nanophrase given by removing $A$ from
 $p_{AB}$, and $p_B$ is the nanophrase given by removing $A$ from
 $p_{AB}$.
\par
By definition,
\begin{equation*}
L(p_{\singular{A}\singular{B}}) = L(p_{AB}) - L(p_A) - L(p_B) + L(p),
\end{equation*}
so we just need to show that
\begin{equation*}
L(p_{AB}) - L(p_A) - L(p_B) + L(p) = \zeromatrix
\end{equation*}
where $\zeromatrix$ is the zero matrix.
In fact, we show that for each $i$ and each $j$ we have
\begin{equation}\label{eqn:linkingmatrix-ft}
l_{ij}(p_{AB}) - l_{ij}(p_A) - l_{ij}(p_B) + l_{ij}(p) = 0.
\end{equation}
\par
If $i$ is equal to $j$, then Equation \eqref{eqn:linkingmatrix-ft}
 obviously holds.
Now assume that $i$ is not equal to $j$.
If $A$ is not in $\mathcal{A}_{ij}(p_{AB})$, then $l_{ij}(p_{AB})$ is
 equal to $l_{ij}(p_B)$ and $l_{ij}(p_A)$ is equal to $l_{ij}(p)$, which
 implies that Equation  \eqref{eqn:linkingmatrix-ft} holds.
Similarly, if $B$ is not in $\mathcal{A}_{ij}(p_{AB})$,
 Equation \eqref{eqn:linkingmatrix-ft} also holds.
The last case to consider is where both $A$ and $B$ are in
 $\mathcal{A}_{ij}(p_{AB})$.
In this case we have 
\begin{align*}
l_{ij}(p_{AB}) = & \xType{A} + \xType{B} + l_{ij}(p), \\
l_{ij}(p_A) = & \xType{A} + l_{ij}(p) \text{ and} \\
l_{ij}(p_B) = & \xType{B} + l_{ij}(p).
\end{align*}
Then Equation \eqref{eqn:linkingmatrix-ft} holds.
\par
We have shown that the linking matrix is a finite type invariant of
 degree less than or equal to $1$.
However, the linking matrix is a non-trivial invariant, so it cannot
 have degree $0$.
\end{proof}
An \emph{orientation} of $\alpha$ is a subset of $\alpha$ which
intersects with each orbit of $\alpha$ under $\tau$ in exactly one
element.
Let $\alpha_o$ be an orientation of $\alpha$.
Then, with respect to $\alpha_o$, any element $g$ of $\pi$ can be
represented uniquely as
\begin{equation*}
g = \sum_{a \in \alpha_o} c_a a,
\end{equation*}
where $c_a$ is in $\Z$ if $a$ equals $\tau(a)$ and in $\cyclic{2}$
otherwise.
We call $c_a$ the coefficient of the $a$ term in $g$.
\par
Let $g_{i,j,a}$ denote the nanophrase of the form
$\dotso|A|\dotso|A|\dotso$, where $i$ is the number of the component in
which the first $A$ occurs, $j$ is the number of the component in which
the second $A$ occurs and $a$ is $\xType{A}$.
For a nanophrase $p$ and an element $a$ in $\alpha_o$, let $l_{ija}(p)$
be the coefficient of the $a$ term in $l_{ij}(p)$. 
If $a$ is not equal to $\tau(a)$, $l_{ija}(p)$ is in $\Z$ and is given
by the following angle bracket formula:
\begin{equation*}
l_{ija}(p) = \langle g_{i,j,a} - g_{i,j,\tau(a)}, p \rangle.
\end{equation*}
If $a$ is equal to $\tau(a)$, $l_{ija}(p)$ is in $\cyclic{2}$ and is
given by the following angle bracket formula:
\begin{equation*}
l_{ija}(p) = \langle g_{i,j,a}, p \rangle \mod 2.
\end{equation*}
\begin{rem}
Since the linking matrix can be calculated from degree $1$ angle bracket
 formulae and the linking matrix is a non-trivial invariant (which
 implies it cannot be a finite type invariant of degree $0$), we could
 have proved Theorem~\ref{thm:linkingmatrix-deg1} just by using
 Proposition~\ref{prop:angleisft}.
\end{rem} 
The following theorem states that the linking matrix essentially
contains all finite type invariants of degree $1$.
\begin{thm}\label{thm:deg1-lm}
Let $p$ and $q$ be two $r$-component nanophrases.
If there exists a finite type invariant $v$ of degree $1$ such that
 $v(p)$ is not equal to $v(q)$, $p$ and $q$ can be distinguished by the
 linking matrix.
\end{thm}
\begin{proof}
We calculate the group $G_1(\alpha,\tau,S,r)$ for arbitrary $\alpha$,
 $\tau$, $S$ and $r$.
By Proposition~\ref{prop:zh-decomp}, $G_1(\alpha,\tau,S,r)$ has the form
\begin{equation*}
G_1(\alpha,\tau,S,r) \cong \Z \oplus H_1(\alpha,\tau,S,r).
\end{equation*}
In $H_1(\alpha,\tau,S,r)$, any nanophrase of rank $2$ or more is equal
 to $0$ by definition.
Rank $1$ nanophrases over $\alpha$ either have the form
 $\dotso|AA|\dotso$ or the form $\dotso|A|\dotso|A|\dotso$.
By relations of the first type, generators of the first form must be
 equal to zero, so we may delete them.
The generators of the second form are the nanophrases $g_{i,j,a}$.
\par
Note that relations of the third type only involve generators that have
 two or more letters.
Since all these generators are zero, these relations trivially hold in
 $H_1(\alpha,\tau,S,r)$ and so we may delete them.
\par
So we have a presentation for $H_1(\alpha,\tau,S,r)$ where the
 generators are the set of $g_{i,j,a}$, for all $i$ and $j$ between $1$
 and $r$ ($i$ less than $j$) and for all $a$ in $\alpha$.
The relations, which are all derived from relations of the second type,
 are
\begin{equation*}
g_{i,j,a} + g_{i,j,\tau(a)}= 0
\end{equation*}
for all $i$ and $j$ between $1$ and $r$ ($i$ less than $j$) and for
 all $a$ in $\alpha$.
Then for any $a$ in $\alpha$ for which $a$ is equal to $\tau(a)$,
the generators $g_{i,j,a}$ (for all $i$ and $j$, $i$ less than $j$) have
 order $2$ in $G_1(\alpha,\tau,S,r)$.
Any other generator has infinite order.
\par
It is easy to check that 
\begin{equation*}
H_1(\alpha,\tau,S,r) \cong \bigoplus_{1 \leq i < j \leq r} H_{ij}
\end{equation*} 
where each $H_{ij}$ is isomorphic to $\pi$.
Recalling that the linking matrix is symmetrical and that the elements
 on the leading diagonal are all $0$, we see that the linking matrix of
 a nanophrase can also be considered as an element of
 $H_1(\alpha,\tau,S,r)$.
\par
Now $G_1(\alpha,\tau,S,r)$ is isomorphic to
\begin{equation*}
\Z \oplus H_1(\alpha,\tau,S,r).
\end{equation*}
Thus we can consider the map $\Gamma_{n,r}$ defined in
 Section~\ref{sec:universal} as mapping an $r$-component nanophrase $p$
 to a pair $(c(p),h(p))$ where $c(p)$ is in $\Z$ and $h(p)$ is in
 $H_1(\alpha,\tau,S,r)$.
Define $\Gamma^\prime$ to be the map from $\Z P(\alpha,\tau,S,r)$
 to $H_1(\alpha,\tau,S,r)$ which maps $p$ to $h(p)$.
Considering the linking matrix as a map from $\Z P(\alpha,\tau,S,r)$
 to $H_1(\alpha,\tau,S,r)$, it is easy to check that
 $\Gamma^\prime$ is equivalent to the linking matrix.
That is, for two $r$-component nanophrases $p$ and $q$,
 $\Gamma^\prime(p)$ is equal to $\Gamma^\prime(q)$ if and only if the
 linking matrices of $p$ and $q$ are the same.
Thus if $v$ is a finite type invariant of degree $1$ which distinguishes
 $p$ and $q$, it must be the case that $\Gamma^\prime(p)$ is not equal
 to $\Gamma^\prime(q)$ and so the linking matrices of $p$ and $q$
 differ.
\end{proof}
\begin{cor}\label{cor:nwft-nodegree1}
For any homotopy of nanowords there are no finite type invariants of
 degree $1$.
\end{cor}
We now give some examples of angle bracket formulae which have degree
$2$ but which give finite type invariants of degree $1$.
We first define some nanophrases to be used in the examples.
Let $i$ and $j$ be integers such that $i$ is less than $j$ and let $a$
and $b$ be (possibly equal) elements in $\alpha$.
Let $e_{i,j,a,b}$ be the nanophrase where the $i$th component has the form
 $AB$, the $j$ has form $AB$, all other components are empty,
 $\xType{A}$ is $a$ and $\xType{B}$ is $b$.
Let $f_{i,j,a,b}$ be the nanophrase where the $i$th component has the form
 $AB$, the $j$ has form $BA$, all other components are empty,
 $\xType{A}$ is $a$ and $\xType{B}$ is $b$.
\begin{ex}\label{ex:angle-ft-degree-a}
Suppose $a$ is not equal to $\tau(a)$ and write $b$ for $\tau(a)$.
We define $l^\prime_{ija}$ by
\begin{equation*}
\begin{split}
l^\prime_{ija}(p) = \langle 
2 e_{i,j,a,a} + 2 f_{i,j,a,a}
- 2 e_{i,j,a,b} - 2 f_{i,j,a,b} 
- 2 e_{i,j,b,a} - 2 f_{i,j,b,a} \\
+ 2 e_{i,j,b,b} + 2 f_{i,j,b,b}
+ g_{i,j,a} + g_{i,j,b}
, p \rangle.
\end{split}
\end{equation*}
For any nanophrase $p$ we claim that $l^\prime_{ija}(p)$ is equal to
 $(l_{ija}(p))^2$ and so $l^\prime_{ija}$ gives a finite type invariant
 of degree $1$.
\par
To prove the claim, we fix a nanophrase $p$ and then let $k$ equal
 $\langle g_{i,j,a},p \rangle$ and let $l$ equal 
$\langle g_{i,j,b},p \rangle$.
Then, by definition, $l_{ija}(p)$ is equal to $k - l$.
One can easily check that the following identities hold:
\begin{equation*}
\langle e_{i,j,a,a} + f_{i,j,a,a}, p \rangle = \frac{1}{2}k(k-1),
\end{equation*}
\begin{equation*}
\langle e_{i,j,a,b} + f_{i,j,a,b} + e_{i,j,b,a} + f_{i,j,b,a}, p \rangle = kl
\end{equation*}
and
\begin{equation*}
\langle e_{i,j,b,b} + f_{i,j,b,b}, p \rangle = \frac{1}{2}l(l-1).
\end{equation*}
As an example we check the first equation. Let $\mathcal{A}_{ij}$ be the
 set of letters which appear both in the $i$th and $j$th components and
 project to $a$.
Note that $\mathcal{A}_{ij}$ contains $k$ letters.
The subphrases of $p$ which contribute to 
$\langle e_{i,j,a,a} + f_{i,j,a,a}, p \rangle$
are exactly those which contain two letters in $\mathcal{A}_{ij}$.
There are $\frac{1}{2}k(k-1)$ different ways to pick an unordered
 pair of letters in $\mathcal{A}_{ij}$ and each such pair contributes
 exactly $1$ to $\langle e_{i,j,a,a} + f_{i,j,a,a}, p \rangle$.
The other equations can be checked in a simlar way.
\par
Since $l^\prime_{ija}$ is a linear combination of the angle
 bracket formulae on the left hand side of the above equations,
 $l^\prime_{ija}$ can be written in terms of $k$ and $l$.
In fact, 
\begin{equation*}
l^\prime_{ija}(p) = (k-l)^2 = (l_{ija}(p))^2.
\end{equation*} 
\end{ex}
\begin{ex}\label{ex:angle-ft-degree-b}
Suppose $a$ is equal to $\tau(a)$.
For a nanophrase $p$, let $k$ equal $\langle g_{i,j,a},p \rangle$.
Then $k$ is the number of letters which appear both in the $i$th and
 $j$th components of $p$ and which project to $a$.
Recall that because $a$ equals $\tau(a)$,
\begin{equation*}
l_{ija}(p) = \langle g_{i,j,a},p \rangle \mod 2
\end{equation*}
 by definition.
Thus $l_{ija}(p)$ is equal to $1$ if $k$ is odd and equal to $0$ if
 $k$ is even.
\par
Consider the angle bracket formula
\begin{equation*}
l^{\prime\prime}_{ija}(p) = \langle 2 e_{i,j,a,a} + 2 f_{i,j,a,a} +
 g_{i,j,a}, p \rangle \mod 4
\end{equation*}
which takes values in $\cyclic{4}$.
We claim that this gives a finite type invariant of degree $1$.
By similar analysis to that in Example~\ref{ex:angle-ft-degree-b} one
 can check that 
\begin{equation*}
\langle 2 e_{i,j,a,a} + 2 f_{i,j,a,a}, p \rangle = k(k-1).
\end{equation*}
Thus
\begin{equation*}
l^{\prime\prime}_{ija}(p) = k^2 \mod 4.
\end{equation*} 
If $k$ is odd, $l^{\prime\prime}_{ija}(p)$ equals $1$ and if $k$ is
 even, $l^{\prime\prime}_{ija}(p)$ equals $0$.
If we consider $l_{ija}(p)$ and $l^{\prime\prime}_{ija}(p)$ as maps from
 the set of $r$-component nanophrases to the set $\{0,1\}$ then we have
 shown that
\begin{equation}\label{eqn:lprimeprime}
l_{ija}(p) = l^{\prime\prime}_{ija}(p)
\end{equation}
for all $r$-component nanophrases $p$.
\par
As a final remark, we note that $l^{\prime\prime}_{ija}$ is surjective
 on $\cyclic{4}$ and so Equation~\eqref{eqn:lprimeprime} does not hold
 over all of $\Z P(\alpha,\tau,S,r)$.
For example, consider the element $m g_{i,j,a}$ in 
$\Z P(\alpha,\tau,S,r)$ for any integer $m$.
Then
\begin{equation*}
l^{\prime\prime}_{ija}(m g_{i,j,a}) = m \mod 4
\end{equation*}
but
\begin{equation*}
l_{ija}(m g_{i,j,a}) = m \mod 2.
\end{equation*}
\end{ex}
\section{Finite type invariants of degree 2}
We now consider finite type invariants of degree 2.
We start by defining some simple examples of such invariants.
\par
Let $p$ be an $n$-component nanophrase.
Let $\alpha_o$ be an orientation of $\alpha$.
Let $i$ and $j$ be integers between $1$ and $n$ inclusive.
Let $a$ and $b$ be elements of $\alpha_o$.
Let $p_{i,i,a,b}$ be the nanophrase where the only non-empty component
is the $i$th component, which is $ABAB$.
If $i$ is not equal to $j$, let $p_{i,j,a,b}$ be the nanophrase where
the only non-empty components are the $i$th component, which is $ABA$,
and the $j$th component, which is $B$.
In either case, $\xType{A}$ is $a$ and $\xType{B}$ is $b$.
Let $u_{i,j,a,b}$ be the map from nanophrases to either $\Z$ or
$\cyclic{2}$ defined as follows
\begin{equation*}
u_{i,j,a,b}(p) =
\begin{cases}
\langle p_{i,j,a,b} - p_{i,j,\tau(a),b} - p_{i,j,a,\tau(b)} +
 p_{i,j,\tau(a),\tau(b)}, p \rangle
& \text{if } a \neq \tau(a), b \neq \tau(b) \\
\langle p_{i,j,a,b} - p_{i,j,\tau(a),b}, p \rangle \mod 2
& \text{if } a \neq \tau(a), b = \tau(b) \\
\langle p_{i,j,a,b} - p_{i,j,a,\tau(b)}, p \rangle \mod 2
& \text{if } a = \tau(a), b \neq \tau(b) \\
\langle p_{i,j,a,b}, p \rangle \mod 2
& \text{if } a = \tau(a), b = \tau(b).
\end{cases}
\end{equation*}
\begin{prop}
Let $(\alpha,\tau,S)$ be homotopy data where $S$ is diagonal.
If $i$ is not equal to $j$ or $a$ is not equal to $b$, $u_{i,j,a,b}$ is
 a finite type invariant of degree $2$.
\end{prop}
\begin{proof}
We first prove invariance.
It is sufficient to prove that for two nanophrases $p$ and $p^\prime$
 related by an isomorphism or a homotopy move, $u_{i,j,a,b}(p)$
 is equal to $u_{i,j,a,b}(p^\prime)$.
It is easy to see that this holds for an isomorphism or an H1
 move.
\par
Consider the case where $p$ and $p^\prime$ are related by an H2
 move.
Then, without loss of generality, we may assume $p$ has the form
 $xCDyDCz$, where $\xType{C}$ is equal to $\tau(\xType{D})$ and
 $p^\prime$ has the form $xyz$.
First note that the subphrase of $p$ just consisting of $C$ and $D$ does
 not match any $p_{i,j,a,b}$.
Let $E$ be some letter in $p$ other than $C$ or $D$.
If the subphrase of $p$ just consisting of $C$ and $E$ matches
 $p_{i,j,a,b}$ for some $i$, $j$, $a$ and $b$, then the subphrase of
 $p$ just consisting of $D$ and $E$ either matches
 $p_{i,j,\tau(a),b}$ or $p_{i,j,a,\tau(b)}$.
In either case, the contribution of the two subphrases to
 $u_{i,j,a,b}(p)$ is $0$.
Any subphrase of $p$ which matches $p_{i,j,a,b}$ for some $i$, $j$, $a$
 and $b$ and does not contain $C$ or $D$ also appears in $p^\prime$.
Thus $u_{i,j,a,b}(p)$ equals $u_{i,j,a,b}(p^\prime)$.
\par
Now consider the case where $p$ and $p^\prime$ are related by an
 H3 move.
Without loss of generality, we can assume that $p$ has the form
 $xCDyCEzDEt$, where $\xType{C}$, $\xType{D}$ and $\xType{E}$ are all
 equal.
First note that any subphrase of $p$ which matches $p_{i,j,a,b}$ for
 some $i$, $j$, $a$ and $b$ and contain less than two of $C$, $D$ or $E$
 also appears in $p^\prime$.
Now consider rank $2$ subphrases of $p$ which contain exactly two of
 $C$, $D$ or  $E$.
There are three such subphrases which we label $p_{CD}$, $p_{CE}$ and
 $p_{DE}$, where the subscript shows which letters appear in the
 subphrase.
In the same way we define $p^\prime_{CD}$, $p^\prime_{CE}$ and
 $p^\prime_{DE}$ as subphrases of $p^\prime$ which has the form
 $xDCyECzEDt$.
We now check how these subphrases contribute to $u_{i,j,a,b}$.
Depending on whether the parts of the H3 move $CD$, $CE$ and $DE$ appear
 in the same or different components, we have four different cases to
 check.
\par
Firstly, there is the case that the three parts of the H3 move, $CD$,
 $CE$ and $DE$, all appear in the $i$th component for some $i$.
Then $p_{CD}$ and $p_{DE}$ both match $p_{i,i,a,a}$ but $p_{CE}$ does
 not match anything.
So these subphrases contribute to $u_{i,i,a,a}(p)$ for which we do
 not claim invariance.
Indeed, $p^\prime_{CE}$ matches $p_{i,i,a,a}$ but $p^\prime_{CD}$ and
 $p^\prime_{DE}$ do not match anything.
Therefore $u_{i,i,a,a}(p)$ is not invariant.
\par
Secondly, there is the case where the first two parts of the H3 move,
 $CD$ and $CE$, are in the $i$th component and the last part, $DE$ is in
 the $j$th component, for some $i$ and $j$ with $i$ less than $j$.
In this case $p_{CD}$ matches $p_{i,j,a,a}$ but $p_{CE}$ and $p_{DE}$ do
 not match anything.
On the other hand, $p^\prime_{CE}$ matches $p_{i,j,a,a}$ but
 $p^\prime_{CD}$ and $p^\prime_{DE}$ do not match anything.
Thus $u_{i,j,a,a}(p)$ equals $u_{i,j,a,a}(p^\prime)$.
\par
Thirdly, there is the case where the first part of the H3 move, $CD$ is
 in the $i$th component and the other two parts, $CE$ and $DE$ are in
 the $j$th component, for some $i$ and $j$ with $i$ less than $j$.
In this case $p_{DE}$ matches $p_{j,i,a,a}$ but $p_{CD}$ and $p_{CE}$ do
 not match anything.
On the other hand, $p^\prime_{CE}$ matches $p_{j,i,a,a}$ but
 $p^\prime_{CD}$ and $p^\prime_{DE}$ do not match anything.
Thus $u_{j,i,a,a}(p)$ equals $u_{j,i,a,a}(p^\prime)$.
\par
Fourthly, there is the case where each part of the H3 move is in a
 different component.
However, in this case, none of $p_{CD}$, $p_{CE}$, $p_{DE}$,
 $p^\prime_{CD}$, $p^\prime_{CE}$ or $p^\prime_{DE}$ match any
 $p_{i,j,a,b}$.  
\par
Therefore $u_{i,j,a,b}$ is invariant under homotopy unless $i$ equals
 $j$ and $a$ equals $b$.
\par
We now show that the $u_{i,j,a,b}$ are finite type invariants of degree
 $2$.
By Proposition~\ref{prop:angleisft} each $u_{i,j,a,b}$ is a finite type
 invariant of degree less than or equal to $2$.
Since $u_{i,j,a,b}(\trivialp{n})$ equals $0$ but
 $u_{i,j,a,b}(p_{i,j,a,b})$ equals $1$, $u_{i,j,a,b}$ is non-trivial and 
 so has degree greater than $0$.
\par
Consider the nanophrase $q$ given by $B|B$ where $\xType{B}$ is $b$.
Then the linking matrices of $p_{1,2,a,b}$ and $q$ are identical.
On the  other hand, $u_{1,2,a,b}(p_{1,2,a,b})$ is equal to $1$ and
 $u_{1,2,a,b}(q)$ is equal to $0$.
Thus by, Theorem~\ref{thm:deg1-lm}, $u_{i,j,a,b}$ is not a finite type
 invariant of degree $1$.
So we can conclude that $u_{i,j,a,b}$ is a finite type invariant of
 degree $2$.
\end{proof}
We will show that Fukunaga's $T$ invariant \cite{Fukunaga:nanophrases2}
is a finite type invariant of degree $2$.
Here we give a slightly different definition to that appearing in
\cite{Fukunaga:nanophrases2}.
However, it is easy to check that the two definitions are equivalent.
\par
Let $p$ be an $n$-component nanophrase and let $\mathcal{A}(p)$ be the
set of letters appearing in $p$.
Define a map $n_p$ from $\mathcal{A}(p) \times \mathcal{A}(p)$ to
$\{-1,0,1\}$ as follows.
Set $n_p(X,X)$ to be $0$ for all letters $X$.
For distinct letters $X$ and $Y$ set
\begin{equation*}
n_p(X,Y) =
\begin{cases}
1  & \text{ if $X$ and $Y$ appear alternating in $p$, starting with $X$} \\
-1 & \text{ if $X$ and $Y$ appear alternating in $p$, starting with $Y$} \\
0  & \text{ if $X$ and $Y$ do not appear alternating}.
\end{cases}
\end{equation*}
Next, for any element $a$ in $\alpha$, define a map $\varepsilon_a$ from
$\mathcal{A}(p)$ to $\{-1,0,1\}$ by
\begin{equation*}
\varepsilon_a(X) =
\begin{cases}
1  & \text{ if } \xType{X} = a \\
-1 & \text{ if } \xType{X} = \tau(a) \neq a \\
0  & \text{ otherwise}.
\end{cases}
\end{equation*}
Then for two elements $a$ and $b$ of $\alpha$ and a letter in
$\mathcal{A}(p)$, define $t_p(a,b,X)$ by
\begin{equation*}
t_p(a,b,X) = \sum_{Y \in \mathcal{A}} 
\varepsilon_a(X)\varepsilon_b(Y) n_p(X,Y).
\end{equation*}
Let $\mathcal{A}_i(p)$ be the set of letters in $\mathcal{A}(p)$ for
which both occurences of the letter appear in the $i$th component of
$p$.
For the $i$th component of $p$, $T^i_{a,b}(p)$ is defined by
\begin{equation*}
T^i_{a,b}(p) =
\begin{cases}
\sum_{A \in \mathcal{A}_i(p)} t_p(a,b,A) \in \Z &
 \text{ if } a \neq \tau(a) \text{ and } b \neq \tau(b) \\
\sum_{A \in \mathcal{A}_i(p)} t_p(a,b,A) \mod 2 \in \cyclic{2} &
 \text{ otherwise}.
\end{cases}
\end{equation*}
Fukunaga proved that when $S$ is diagonal, $T^i_{a,b}(p)$ is a homotopy
invariant \cite{Fukunaga:nanophrases2}.
\par
Let $\alpha_o$ be an orientation of $\alpha$.
Let $T^i(p)$ denote the tuple of elements $T^i_{a,b}(p)$ for all $a$ and
$b$ in $\alpha_o$.
Note that for any $a$ and $b$ in $\alpha$, we can calculate
$T^i_{a,b}(p)$ from $T^i(p)$ because we have the following
relations
\begin{equation*}
T^i_{a,b}(p) = - T^i_{\tau(a),b}(p) = - T^i_{a,\tau(b)}(p) =
 T^i_{\tau(a),\tau(b)}(p),
\end{equation*}
which can be derived from the definition.
Fukunaga's $T$ invariant is the $n$-tuple consisting of the $T^i(p)$.
\begin{thm}\label{thm:t-degree2}
The $T$ invariant is a degree $2$ finite type invariant.
\end{thm}
\begin{proof}
Each $T^i_{a,b}(p)$ can be written as a linear combination of
the invariants $u_{i,j,a,b}$ or $u_{j,i,b,a}$:
\begin{equation*}
T^i_{a,b}(p) = u_{i,i,a,b}(p) - u_{i,i,b,a}(p)
 + \sum_{j=1}^{i-1} u_{j,i,b,a}(p) 
 + \sum_{j=i+1}^{n} u_{i,j,a,b}(p)
\end{equation*}
where the sum is calculated in $\Z$ or $\cyclic{2}$ appropriately.
Therefore $T^i_{a,b}(p)$ is a finite type invariant of degree less than
 or equal to 2.
\par
On the other hand, $T^i_{a,b}(p)$ is non-trivial and so has degree greater
 than $0$.
For homotopy of Gauss phrases (nanophrases over a set containing a
 single element, where $\tau$ is the identity map and $S$ is diagonal),
 the first author showed that Fukunaga's $T$ and the linking matrix are 
 independent \cite{Gibson:gauss-phrase}.
By a simple adaption of the argument, or just by considering the
 projection of nanophrases to Gauss phrases, it is clear that this fact
 holds for any nanophrase homotopy where $S$ is diagonal.
So $T^i_{a,b}(p)$ is not degree $1$ and therefore must be a degree $2$
 finite type invariant. 
\par
As $T$ is essentially a tuple of the $T^i_{a,b}(p)$, it then follows
 that $T$ is a degree $2$ finite type invariant.
\end{proof}
\begin{thm}\label{thm:deg2-tindep}
There exist degree $2$ non-trivial finite type invariants which
 are independent of $T$.
\end{thm}
\begin{proof}
Consider the nanophrases $ABAC|BC|\trivial$ and $ABAC|\trivial|BC$
 where, in both cases, $\xType{A}$ is $a$, $\xType{B}$ is $b$ and
 $\xType{C}$ is $\tau(b)$ for some $a$ and $b$ in $\alpha_o$.
Then the invariant $u_{1,2,a,b}$ can distinguish the two nanophrases but
 the $T$ invariant cannot.
Note that both nanophrases have the same linking matrix, so they cannot
 be distinguished by degree $1$ finite type invariants.  
\end{proof}
\begin{thm}\label{thm:deg2-nanoword}
Given $\alpha$ and $\tau$, let $l$ be the number of free orbits of
 $\alpha$ under $\tau$ and let $k$ be the number of fixed orbits.
When $S$ is empty, the group $G_2(\alpha,\tau,S,1)$ is isomorphic to
\begin{equation}\label{eqn:empty-s}
(\Z)^{l^2+1} \oplus (\cyclic{2})^{k^2+2kl}.
\end{equation}
When $S$ is diagonal, the group $G_2(\alpha,\tau,S,1)$ is isomorphic to
\begin{equation}\label{eqn:diagonal-s}
(\Z)^{l^2-l+1} \oplus (\cyclic{2})^{k^2+2kl-k}.
\end{equation}
When $S$ is $\alpha \times \alpha \times \alpha$, $G_2(\alpha,\tau,S,1)$
 is isomorphic to 
\begin{equation}\label{eqn:full-s}
\Z \oplus (\cyclic{2})^{k+l-1}.
\end{equation}
\end{thm}
\begin{proof}
By Proposition~\ref{prop:zh-decomp}, $G_2(\alpha,\tau,S,1)$ has the form
\begin{equation*}
G_2(\alpha,\tau,S,1) \cong \Z \oplus H_2(\alpha,\tau,S,1).
\end{equation*}
We calculate $H_2(\alpha,\tau,S,1)$.
\par
By relations of the fourth type, if a nanoword $w$ has rank greater than
 $2$ then $w$ equals $0$ in $H_2(\alpha,\tau,S,1)$.
By relations of the first type, if $w$ is isomorphic to a word of the
 form $xAAy$ then $w$ equals $0$ in $H_2(\alpha,\tau,S,1)$.
Thus nanowords of these types can be eliminated from the presentation
 of $H_2(\alpha,\tau,S,1)$.
We say that a nanoword $w$ is a non-trivial generator of
 $H_2(\alpha,\tau,S,1)$ if its rank is greater than $0$ and less than or
 equal to $2$ and $w$ is not isomorphic to a nanoword of the form
 $xAAy$.
So non-trivial generators of $H_2(\alpha,\tau,S,1)$ have the form
 $ABAB:ab$ for $a$ and $b$ (possibly equal) elements in $\alpha$.
\par
For any $S$, by relations of the second type, we have
\begin{equation}
ABCBAC:a\tau(a)b + ACAC:ab + BCBC:\tau(a)b = 0
\end{equation}
and
\begin{equation}
ABCACB:ab\tau(b) + ABAB:ab + ACAC:a\tau(b) = 0
\end{equation}
for all elements $a$ and $b$ of $\alpha$.
From these relations, using isomorphisms and relations of the fourth
 type, we derive 
\begin{equation}\label{eqn:trunc2}
ABAB:ab = -ABAB:\tau(a)b = -ABAB:a\tau(b) = ABAB:\tau(a)\tau(b)  
\end{equation}
for all elements $a$ and $b$ of $\alpha$.
Note that if $\tau(a)$ is equal to $a$ or if $\tau(b)$ is equal to $b$,
 then \eqref{eqn:trunc2} implies
\begin{equation}\label{eqn:order2}
2ABAB:ab = 0.
\end{equation}
Let $\alpha_0$ be an orientation of $\alpha$.
Then all non-trivial generators can be written in terms of generators of
 the form $ABAB:ab$ where $a$ and $b$ are in $\alpha_0$.
Thus we can eliminate from the presentation all generators not of this
 form.
We have now considered all relations given by the first relation or the
 second relation.
\par
We now consider the case when $S$ is empty.
In this case, there are no more relations to consider.
By \eqref{eqn:trunc2} we know that generators of the form $ABAB:ab$ and
 $ABAB:cd$ are dependent if and only if $a$ is in the same orbit of
 $\tau$ as $c$ and $b$ is in the same orbit of $\tau$ as $d$.
Thus we have exactly $(k+l)^2$ independent generators.
\par
If $a$ or $b$ are in a fixed orbit of $\tau$, then by
 \eqref{eqn:order2}, the generator $ABAB:ab$ has order $2$.
A generator of this form generates a subgroup isomorphic to $\cyclic{2}$.
There are $k^2+2kl$ independent generators of this type.
\par
On the other hand, if $a$ and $b$ are both in free orbits of $\tau$,
 $ABAB:ab$ generates a subgroup isomorphic to $\Z$.
There are $l^2$ independent generators of this type.
\par
Thus, when $S$ is empty, $G_2(\alpha,\tau,S,1)$ is isomorphic to the
 group in \eqref{eqn:empty-s}.
\par
We now consider the case where $S$ is diagonal.
In this case we get exactly one relation of the third type for each $a$
 in $\alpha$:
\begin{multline*}
ABACBC:aaa + ABAB:aa + AACC:aa + BCBC:aa = \\
BACACB:aaa + BAAB:aa + ACAC:aa + BCCB:aa. 
\end{multline*}
Using isomorphisms and relations of the first and fourth types, this
 simplifies to 
\begin{equation*}
ABAB:aa = 0
\end{equation*}
for all $a$ in $\alpha$.
So we can eliminate generators of the form $ABAB:aa$ from the
 presentation and we are left with generators of the form $ABAB:ab$
 where $a$ and $b$ are in $\alpha_0$ and $a$ is not equal to $b$.
There are $k^2+2kl-k$ generators of this form where at least one of $a$
or $b$ is in a fixed orbit of $\tau$.
These generators have order $2$.
On the other hand, there are $l^2-l$ generators for which $a$ and $b$
are both in free orbits of $\tau$.
Thus, when $S$ is diagonal, $G_2(\alpha,\tau,S,1)$ is isomorphic to the
 group in \eqref{eqn:diagonal-s}.
\par
We now consider the case where $S$ is
$\alpha \times \alpha \times \alpha$.
Then relations of the third type are
\begin{multline*}
ABACBC:abc + ABAB:ab + AACC:ac + BCBC:bc = \\
BACACB:abc + BAAB:ab + ACAC:ac + BCCB:bc
\end{multline*}
for all $a$, $b$ and $c$ in $\alpha$.
Using isomorphisms and relations of the first and fourth types, we
 simplify this to 
\begin{equation}\label{eqn:third-simple}
ABAB:ab + ABAB:bc - ABAB:ac = 0.
\end{equation}
When $b$ and $c$ both equal $a$ this gives $ABAB:aa = 0$ for all $a$ in
 $\alpha$ as in the case where $S$ is diagonal.
\par
Consider equation \eqref{eqn:third-simple} for the triple
 $(\tau(a),b,c)$:
\begin{equation*}
ABAB:\tau(a)b + ABAB:bc - ABAB:\tau(a)c = 0.
\end{equation*}
By \eqref{eqn:trunc2} this becomes
\begin{equation}\label{eqn:third-tau}
-ABAB:ab + ABAB:bc + ABAB:ac = 0.
\end{equation}
Adding \eqref{eqn:third-simple} and \eqref{eqn:third-tau} we get
\begin{equation*}
2 ABAB:bc = 0
\end{equation*}
for all $b$ and $c$ in $\alpha$.
Thus all generators are either equal to $0$ or have order $2$.
Using this fact we simplify \eqref{eqn:third-simple} to
\begin{equation}\label{eqn:third-simplest}
ABAB:ab + ABAB:ac + ABAB:bc = 0.
\end{equation}
\par
If $a$ equals $b$, \eqref{eqn:third-simplest} becomes
\begin{equation*}
ABAB:aa + ABAB:ac + ABAB:ac = 0
\end{equation*}
which is trivially true.
If $b$ equals $c$, \eqref{eqn:third-simplest} becomes
\begin{equation*}
ABAB:ab + ABAB:ab + ABAB:bb = 0
\end{equation*}
which is also trivially true.
If $a$ equals $c$, \eqref{eqn:third-simplest} becomes
\begin{equation*}
ABAB:ab + ABAB:aa + ABAB:ba = 0
\end{equation*}
which implies
\begin{equation*}
ABAB:ab = ABAB:ba
\end{equation*}
for all $a$ and $b$ in $\alpha$.
\par
We pick an order on the elements of $\alpha_0$.
Then any term can be written in terms of generators of the form
$ABAB:ab$ where $a$ and $b$ are both in $\alpha_0$ and $a$ is less than
$b$.
The only remaining relations are those of the form in
\eqref{eqn:third-simplest} where $a$ is less than $b$ and $b$ is less
than $c$.
\par
When $\alpha_0$ has less than three elements, no relations remain.
When $\alpha_0$ contains only one element, all generators are equal to
$0$.
When $\alpha_0$ contains exactly two elements $a$ and $b$ ($a$ less than
 $b$), we have just one non-zero generator $ABAB:ab$.
\par
When $\alpha_0$ contains exactly three elements $a$, $b$ and $c$ ($a$
 less than $b$ and $b$ less than $c$), we have just three non-zero
 generators $ABAB:ab$, $ABAB:ac$ and $ABAB:bc$.
Only the relation \eqref{eqn:third-simplest} remains.
We use to eliminate the generator $ABAB:bc$ and we are left with two
 independent generators.
\par
Let $m$ be the number of elements in $\alpha_0$ ($m$ equals $l+k$).
We claim that the number of independent generators is equal to $m-1$.
We prove this by induction on $m$.
By the above discussion we have already seen that this is true when $m$
is $1$, $2$ or $3$.
Now, assuming that the result is true for $m-1$, we prove the result
true for $m$ greater than or equal to $4$.
\par
Let $a$ be the first letter in $\alpha_0$ and $x$ be the last letter in
$\alpha_0$, according to the order that we assigned to $\alpha_0$.
We will eliminate all relations which contain an $x$.
For each element $b$ in $\alpha_0 - \{a,x\}$ (such an element $b$ exists
 because $m$ is greater than or equal to $4$) we have a relation
\begin{equation*}
ABAB:ab + ABAB:ax + ABAB:bx = 0.
\end{equation*}
Rearranging, we get
\begin{equation}\label{eqn:bx-subs}
ABAB:bx = ABAB:ab + ABAB:ax
\end{equation}
which we use to eliminate generators of the form $ABAB:bx$ ($b$ in
$\alpha_0 - \{a,x\}$) from any other relations in which they appear.
Thus for any two different elements $b$ and $c$ in $\alpha_0 - \{a,x\}$,
($b$ less than $c$) we have the relation
\begin{equation*}
ABAB:bc + ABAB:bx + ABAB:cx = 0.
\end{equation*}
Substituting \eqref{eqn:bx-subs} for $ABAB:bx$ and the corresponding
expression for $ABAB:cx$, we get
\begin{equation*}
ABAB:bc + ABAB:ab + ABAB:ax + ABAB:ac + ABAB:ax = 0.
\end{equation*}
This simplifies to
\begin{equation*}
ABAB:bc + ABAB:ab + ABAB:ac = 0
\end{equation*}
which is equivalent to a relation that we already have.
In this way we can eliminate all generators involving $x$ from the
relations.
Note that we didn't rewrite $ABAB:ax$ in terms of other generators, but
it no longer appears in any relation.
This gives us an independent generator.
\par
The remaining relations are those for all ordered triples in 
$\alpha_0 - \{x\}$.
By the induction hypothesis, we can solve these equations to find $m-2$
independent generators.
So, including $ABAB:ax$, we have $m-1$ independent generators, as
claimed.
Thus, when $S$ is $\alpha \times \alpha \times \alpha$,
 $G_2(\alpha,\tau,S,1)$ is isomorphic to the 
 group in \eqref{eqn:full-s}.
\end{proof}
\section{Gauss words}\label{sec:gausswords}
Let $\alpha_{GW}$ be the set $\{a\}$, $\tau_{GW}$ be the
identity map and $S_{GW}$ be $\{(a,a,a)\}$.
For any nanoword over $\alpha_{GW}$, all letters map to $a$, so we can
forget the map to $\alpha_{GW}$ and just consider nanowords over
$\alpha$ as Gauss words.
The homotopy given by $(\alpha_{GW},\tau_{GW},S_{GW})$ is called
homotopy of Gauss words (it was called open homotopy of Gauss words in 
\cite{Gibson:gauss-word}).
\par
In \cite{Turaev:Words}, Turaev conjectured that all Gauss words are
homotopic to the trivial Gauss word.
However, the existence of counterexamples was shown independently in
\cite{Gibson:gauss-word} and \cite{Manturov:freeknots}.
In particular, in \cite{Gibson:gauss-word}, the first author showed that
the Gauss word $ABACDCBD$ is such a counterexample.
Later in this section we will show that $ABACDCBD$ is non-trivial using
a finite type invariant of degree $4$.
\par
In this section, we write $G_n$ for
$G_n(\alpha_{GW},\tau_{GW},S_{GW},1)$. 
We have the following proposition.
\begin{prop}\label{prop:no123}
For Gauss words, $G_1$, $G_2$ and $G_3$ are all isomorphic to $\Z$.
Thus there are no finite type invariants of degree $1$, $2$ or
 $3$.
\end{prop}
\begin{proof}
The fact that $G_1$ is isomorphic to $\Z$ follows from
 Corollary~\ref{cor:nwft-nodegree1} and
by Theorem~\ref{thm:deg2-nanoword}, $G_2$ is isomorphic to $\Z$.
We now calculate $G_3$.
\par
By Proposition~\ref{prop:zh-decomp}, $G_3$ has the form
\begin{equation*}
G_3 \cong \Z \oplus H_3
\end{equation*}
where we have written $H_3$ for $H_3(\alpha_{GW},\tau_{GW},S_{GW},1)$.
By relations of the fourth type, if a Gauss word $w$ has rank greater
 than $3$ then $w$ equals $0$ in $H_3$.
By relations of the first type, if $w$ is isomorphic to a word of the
 form $xAAy$ then $w$ equals $0$ in $H_3$.
Thus Gauss words of these types can be eliminated from the presentation
 of $H_3$.
We are left with $6$ generators:
 $ABAB$, $ABACBC$, $ABCABC$, $ABCACB$, $ABCBAC$ and $ABCBCA$.
The following relations are relations of the second type:
\begin{equation*}
ABCACB + ABAB + ACAC = 0
\end{equation*} 
and
\begin{equation*}
ABCBAC + ACAC + BCBC = 0
\end{equation*}
which are equivalent to
\begin{equation*}
ABCACB + 2 ABAB = 0
\end{equation*} 
and
\begin{equation*}
ABCBAC + 2 ABAB = 0.
\end{equation*}
From relations of the second type we also get
\begin{equation*}
2 ABCABC = 0.
\end{equation*}
From relations of the third type we get
\begin{align*}
ABACBC + & ABAB + AACC + BCBC \\
= & BACACB + BAAB + ACAC + BCCB, \\
DABACDBC + & DABADB + DAACDC + DBCDBC \\
= & DBACADCB + DBAADB + DACADC + DBCDCB, \\
DABACBCD + & DABABD + DAACCD + DBCBCD \\
= & DBACACBD + DBAABD + DACACD + DBCCBD
\end{align*}
and
\begin{align*}
ABDACDBC + & ABDADB + ADACDC + BDCDBC \\
= & BADCADCB + BADADB + ADCADC +
 BDCDCB.
\end{align*}
Removing trivial generators and canceling isomorphic words, these
 relations become
\begin{align*}
ABACBC + ABAB & = ABCBCA, \\
ABCABC & = ABCACB, \\
ABCBCA & = 0
\end{align*} 
and
\begin{equation*}
ABCACB + ABACBC + ABCBAC = 2 ABCBCA + ABCABC.
\end{equation*}
\par
Solving all the above relations gives
\begin{equation*}
ABAB = ABACBC = ABCABC = ABCACB = ABCBAC = ABCBCA = 0.
\end{equation*}
Thus $G_3$ is isomorphic to $\Z$.
\end{proof}
We define a map $v_4$ from the set of Gauss words to $\cyclic{2}$ as follows.
For $i$ running from $1$ to $6$, the Gauss words $w_i$ are defined by
\begin{align*}
w_1 & = ABACDCBD, \\  
w_2 & = ABCACDBD, \\
w_3 & = ABCADBDC, \\
w_4 & = ABCBDACD, \\
w_5 & = ABCDBDAC
\end{align*}
and
\begin{equation*}
w_6 = ABCDCADB.
\end{equation*}
Then $v_4$ is given by
\begin{equation*}
v_4(w) = \langle \sum_{i=1}^6 w_i, w \rangle \mod 2.
\end{equation*}
\begin{prop}
The map $v_4$ is a homotopy invariant of Gauss words.
\end{prop}
\begin{proof}
We must prove that if $w$ and $w^\prime$ are two homotopic Gauss words,
 $v_4(w)$ and $v_4(w^\prime)$ are equal.
It is sufficient to prove this in the case that $w$ and $w^\prime$ are
 related by an isomorphism or homotopy move.
In the case that $w$ and $w^\prime$ are related by an isomorphism
 it is clear that $v_4(w)$ equals $v_4(w^\prime)$.
We now consider each homotopy move in turn.
\par
If $w$ and $w^\prime$ are related by an H1 move, then we may
 assume without loss of generality that $w$ has the form $xAAy$ and
 $w^\prime$ has the form $xy$.
Now observe that none of the $w_i$ are isomorphic to a Gauss word of the
 form $uAAv$.
Thus if $w_i$ is isomorphic to $s$, a subword of $w$, $s$ does not
 contain the letter $A$ and so $s$ is also a subword of $w^\prime$.
Thus $\langle w_i, w \rangle$ is equal to 
$\langle w_i, w^\prime \rangle$ and so $v_4(w)$ equals $v_4(w^\prime)$.
\par
If $w$ and $w^\prime$ are related by an H2 move, then we may
 assume without loss of generality that $w$ has the form $xAByBAz$ and
 $w^\prime$ has the form $xyz$.
Now observe that none of the $w_i$ are isomorphic to a Gauss word of the
 form $tABuBAv$.
If $w_i$ is isomorphic to $s$, a subword of $w$, either $s$ does not
 contain the letters $A$ and $B$, or it contains exactly one of them.
If $s$ does not contain $A$ and $B$, then $s$ is a subword of
 $w^\prime$.
Suppose that $s$ contains one of $A$ or $B$.
Without loss of generality we assume that it contains $A$.
Then there exists a subword $s^\prime$ of $w$ which contains the same
 letters as $s$ except that the letter $A$ is replaced with a $B$.
Clearly $s^\prime$ is also isomorphic to $w_i$.
Thus subwords isomorphic to $w_i$ and containing one of $A$ or $B$
 appear in pairs.
As $v_4(w)$ is defined modulo $2$, these pairs do not contribute anything
 to $v_4(w)$.
Thus $v_4(w)$ equals $v_4(w^\prime)$.
\par
If $w$ and $w^\prime$ are related by an H3 move, then we may
 assume without loss of generality that $w$ has the form $xAByACzBCt$
 and $w^\prime$ has the form $xBAyCAzCBt$.
Suppose $s$ is a subword of $w$ which is isomorphic to a $w_i$.
Let $m(s)$ be the number of letters in the set $\{A,B,C\}$ which appear
 in $s$.
For each $m$ in $\{0,1,2,3\}$, we will show that the contributions to
 $v_4(w)$ of subwords $s$ of $w$ with $m(s)$ equal to $m$ matches the
 contributions to $v_4(w^\prime)$ of subwords $s^\prime$ of $w^\prime$
 with $m(s^\prime)$ equal to $m$.
\par
First consider the case where $m$ is $0$.
Then any subword $s$ of $w$ which $m(s)$ equal to $0$ is also a subword
 of $w^\prime$.
Since these subwords are in one-to-one correspondence, their
 contributions to $v_4(w)$ and $v_4(w^\prime)$ are equal.
\par
Next consider the case where $m$ is $1$.
It is easy to see that any subword $s$ of $w$ which $m(s)$ equal to $1$
 is also a subword of $w^\prime$.
Thus the contributions to $v_4(w)$ and $v_4(w^\prime)$ of these kinds of
 subwords are also equal.
\par
Now consider the case where $m$ is $2$.
Suppose $s$ is a rank $4$ subword of $w$ such that $m(s)$ is equal to
 $2$.
Then $s$ contains two letters from $\{A,B,C\}$ and also contains another
 two letters $D$ and $E$.
So we consider all rank $5$ Gauss words which contain $ABACBC$ as a
 subword and check how the subwords of rank $4$ which contain exactly
 two letters from $\{A,B,C\}$ contribute to $v_4(w)$.
We note that since none of the words $w_i$ are of the form $uDDv$,
 $tDEuDEv$ or $tDEuEDv$, we can elimate words of this form.
The remaining words are listed in Table~\ref{tab:m2invariance}.
To derive the subwords of these words which contain exactly two letters
 from $\{A,B,C\}$ we must delete one of $A$, $B$ or $C$ from the word.
This gives $3$ subwords.
In column $A$ we indicate which of the $w_i$ is isomorphic to the
 subword derived by deleting $A$.
If the subword is not isomorphic to any of the $w_i$, the column is left
 blank.
The columns $B$ and $C$ indicate the results of deleting $B$ and $C$
 respectively.
The next column, labelled ``word after H3'', shows the result of
 applying the H3 move involving $A$, $B$ and $C$ on the word.
The columns $A$, $B$ and $C$ to the right of this column indicate the
 result of deleting $A$, $B$ or $C$ from this second word.
From the table, it is clear that in each case the contributions to
 $v_4(w)$ and $v_4(w^\prime)$ are equivalent modulo $2$.
\begin{table}[hbt]
\begin{center}
\begin{tabular}{c|c|c|c|c|c|c|c}
word & $A$ & $B$ & $C$ &
word after H3 & $A$ & $B$ & $C$ \\
\hline
$DABDACEBCE$ & & & & $DBADCAECBE$ & $w_1$ & & $w_2$ \\
$DABEACDBCE$ & & & & $DBAECADCBE$ & & & \\
$DABEACEBCD$ & & & & $DBAECAECBD$ & & & \\
\hline
$DEABDACEBC$ & & $w_2$ & & $DEBADCAECB$ & $w_3$ & & \\
$DEABEACDBC$ & $w_4$ & & & $DEBAECADCB$ & & & $w_5$ \\
$DEABDACBCE$ & & & & $DEBADCACBE$ & & & \\
$DEABEACBCD$ & & & & $DEBAECACBD$ & & & \\
$DEABACDBCE$ & & & $w_6$ & $DEBACADCBE$ & & $w_6$ & \\
$DEABACEBCD$ & & & & $DEBACAECBD$ & & & \\
$DABDEACEBC$ & $w_1$ & & $w_3$ & $DBADECAECB$ & & & \\
$DABEDACEBC$ & & & & $DBAEDCAECB$ & & & \\
$DABDEACBCE$ & & & & $DBADECACBE$ & & & \\
$DABEDACBCE$ & $w_3$ & & & $DBAEDCACBE$ & & $w_3$ & \\
$ABDEACDBCE$ & & $w_3$ & & $BADECADCBE$ & $w_5$ & & \\
$ABDEACEBCD$ & $w_6$ & & & $BADECAECBD$ & & & $w_5$ \\
$DABEACDEBC$ & & $w_4$ & & $DBAECADECB$ & & & $w_6$ \\
$DABEACEDBC$ & & & $w_5$ & $DBAECAEDCB$ & $w_6$ & & \\
$DABACDEBCE$ & & & & $DBACADECBE$ & & & \\
$DABACEDBCE$ & & & $w_4$ & $DBACAEDCBE$ & & $w_4$ & \\
$ABDACDEBCE$ & $w_4$ & & $w_2$ & $BADCADECBE$ & & & \\
$ABDACEDBCE$ & & & & $BADCAEDCBE$ & & & \\
$DABEACBCDE$ & $w_5$ & & & $DBAECACBDE$ & & $w_5$ & \\
$DABEACBCED$ & & & & $DBAECACBED$ & & & \\
$DABACEBCDE$ & & & & $DBACAECBDE$ & & & \\
$DABACEBCED$ & & & & $DBACAECBED$ & & & \\
$ABDACEBCDE$ & & $w_1$ & & $BADCAECBDE$ & & & $w_4$ \\
$ABDACEBCED$ & & & $w_3$ & $BADCAECBED$ & $w_6$ & & \\
\hline
$DEDABEACBC$ & & & & $DEDBAECACB$ & & & \\
$DEDABACEBC$ & & & $w_1$ & $DEDBACAECB$ & & $w_1$ & \\
$DEDABACBCE$ & & & & $DEDBACACBE$ & & & \\
$DABEDEACBC$ & $w_2$ & & & $DBAEDECACB$ & & $w_2$ & \\
$ABDEDACEBC$ & & & $w_6$ & $BADEDCAECB$ & & $w_4$ & \\
$ABDEDACBCE$ & & & & $BADEDCACBE$ & & & \\
$DABACEDEBC$ & & & & $DBACAEDECB$ & & & \\
$ABDACEDEBC$ & $w_5$ & & & $BADCAEDECB$ & & $w_3$ & \\
$ABACDEDBCE$ & & & $w_1$ & $BACADEDCBE$ & & $w_1$ & \\
$DABACBCEDE$ & & & & $DBACACBEDE$ & & & \\
$ABDACBCEDE$ & $w_2$ & & & $BADCACBEDE$ & & $w_2$ & \\
$ABACDBCEDE$ & & & & $BACADCBEDE$ & & & \\
\end{tabular}
\end{center}
\caption{Contributions to $v_4$ when $m(s)$ is $2$}
\label{tab:m2invariance}
\end{table}
\par
Finally consider the case where $m$ is $3$.
Suppose $s$ is a rank $4$ subword of $w$ such that $m(s)$ is equal to
 $3$.
Let $D$ be the $4$th letter of $s$.
Suppose the two occurences of $D$ in $s$ appear together (that is, $s$
 is of the form $uDDv$).
Then, as we observed above, no $w_i$ is of this form, so $s$ cannot be
 isomorphic to a $w_i$.
Thus it is sufficient to check only the cases where the two occurences
 of $D$ do not appear together.
There are $6$ such cases.
We show that the contributions to $v_4(w)$ and $v_4(w^\prime)$ of these
 kinds of subwords are equal in Table~\ref{tab:m3invariance}.
\begin{table}[hbt]
\begin{center}
\begin{tabular}{c|c|c|c|c|c}
$s$ & matches & $v_4(w)$ &
$s^\prime$ & matches & $v_4(w^\prime)$ \\
\hline
$DABDACBC$ &       & $0$ & $DBADCACB$ &       & $0$ \\
$DABACDBC$ & $w_4$ & $1$ & $DBACADCB$ & $w_6$ & $1$ \\
$DABACBCD$ &       & $0$ & $DBACACBD$ &       & $0$ \\
$ABDACDBC$ &       & $0$ & $BADCADCB$ &       & $0$ \\
$ABDACBCD$ & $w_3$ & $1$ & $BADCACBD$ & $w_5$ & $1$ \\
$ABACDBCD$ &       & $0$ & $BACADCBD$ &       & $0$ \\
\end{tabular}
\end{center}
\caption{Contributions to $v_4$ when $m(s)$ is $3$}
\label{tab:m3invariance}
\end{table}
\par
This completes the proof.
\end{proof}
Now $v_4(\trivial)$ is equal to $0$ but $v_4(ABACDCBD)$ is equal to $1$.
Thus $ABACDCBD$ is a homotopically non-trivial Gauss word and
the invariant $v_4$ is non-trivial.
In fact, $ABACDCBD$ was shown to be homotopically non-trivial in
\cite{Gibson:gauss-word}.
The invariant $v_4$ gives another simple way to prove this fact.
\begin{prop}
The invariant $v_4$ is a finite type invariant of degree $4$.
\end{prop}
\begin{proof}
By Proposition~\ref{prop:angleisft}, $v_4$ is a finite type invariant of
 degree less than or equal to $4$.
As $v_4$ is non-trivial, it is not of degree $0$ and so
 Proposition~\ref{prop:no123} implies $v_4$ must have degree greater than
 or equal to $4$.
\end{proof}
\begin{rem}\label{rem:v4-ftinv}
There is a natural map from open virtual knot diagrams to Gauss words
 and this induces a well-defined map from open virtual knots to homotopy
 classes of Gauss words \cite{Turaev:KnotsAndWords}.
Thus any finite type invariant of Gauss words is a finite type invariant
 of open virtual knots.
In particular $v_4$ is a finite type invariant of open virtual knots.
\end{rem}
\begin{rem}
The virtualization move, shown in Figure~1 in
 \cite{chrisman:ftype-vmove}, is a local diagrammatic move on a single
 real crossing.
Under the natural map from open virtual knot diagrams to Gauss words
 mentioned in Remark~\ref{rem:v4-ftinv}, it is clear that two open virtual knot
 diagrams related by a virtualization move map to the same Gauss word.
\par
In \cite{chrisman:ftype-vmove}, it is stated that there are no
 non-constant Goussarov-Polyak-Viro finite type invariants for virtual
 knots (open or closed) which are invariant under the virtualization
 move. 
This statement appears to contradict Remark~\ref{rem:v4-ftinv}.
However, there is no contradiction because in
 \cite{chrisman:ftype-vmove} only $\Z$-valued finite type invariants are
 considered and $v_4$ is $\cyclic{2}$-valued.
\end{rem}
The following theorem shows that $v_4$ is essentially the only finite
 type invariant of degree $4$ for Gauss words.
\begin{thm}\label{thm:gw}
For Gauss words, $G_4$ is isomorphic to $\Z \oplus \cyclic{2}$.
If two Gauss words, $w$ and $w^\prime$, can be distinguished by a
finite type invariant of degree $4$, they can be distinguished by
 $v_4$.
\end{thm}
\begin{proof}
By a straightforward but lengthy calculation which we omit, it can be
 shown that $G_4$ is the additive abelian group given by
$\langle \trivial, w_1 | 2 w_1 = 0 \rangle$.
\end{proof}
\section{Closed homotopy}\label{sec:closed-homotopy}
In \cite{Turaev:KnotsAndWords} Turaev defined \emph{shift moves} on
nanophrases.
Let $\nu$ be an involution on $\alpha$.
Let $p$ be an $r$-component nanophrase over $\alpha$.
A \emph{shift move} on the $i$th component of $p$ is a move which gives
a new nanophrase $p^\prime$ as follows.
If the $i$th component of $p$ is empty or contains a single letter,
$p^\prime$ is $p$.
If not, the $i$th component of $p$ has the form $Ax$.
Then the $i$th component of $p^\prime$ is $xA$ and for all $j$ not equal
to $i$, the $j$th component of $p^\prime$ is the same as the $j$th
component of $p$.
Furthermore, writing $\xType{A}_p$ for $\xType{A}$ in $p$ and
$\xType{A}_{p^\prime}$ for $\xType{A}$ in $p^\prime$, if $x$ contains
the letter $A$, then $\xType{A}_{p^\prime}$ equals $\nu(\xType{A})$.
Otherwise, $\xType{A}_{p^\prime}$ equals $\xType{A}$.
\begin{ex}
Let $\alpha$ be the set $\{a,b\}$ and $\nu$ be the involution on
 $\alpha$ which swaps $a$ and $b$.
Let $p$ be the nanophrase $ABAC|BC:aaa$.
Applying a shift move to the $1$st component of $p$ gives
 $BACA|BC:baa$.
Applying a shift move to the $2$nd component of $p$ gives
 $ABAC|CB:aaa$.
\end{ex}
\emph{Closed homotopy} of nanophrases over $\alpha$ is the equivalence
relation generated by homotopy and shift moves.
The definition of closed homotopy is parameterized by $\alpha$, $\tau$,
$S$ and $\nu$.
\begin{rem}\label{rem:vknot-homotopy-closed}
Recall from Remark~\ref{rem:vknot-homotopy} the definitions of
 $\alpha_{vk}$, $\tau_{vk}$ and $S_{vk}$.
Let $\nu_{vk}$ be the involution on $\alpha$ which sends $\aplus$ to
 $\bplus$ and $\aminus$ to $\bminus$.
Then $\alpha_{vk}$, $\tau_{vk}$, $S_{vk}$ and $\nu_{vk}$ define a closed
 homotopy.
Turaev showed that under this homotopy, the homotopy classes of
 nanophrases over $\alpha_{vk}$ correspond to ordered virtual links
 (virtual links where the components are ordered and equivalence of
 ordered virtual links respects this order)
 \cite{Turaev:KnotsAndWords}.
\end{rem}
\par
The definition of finite type invariants and universal invariants
extends to closed homotopy.
Indeed our definition for finite type invariants of the homotopy given
in Remark~\ref{rem:vknot-homotopy-closed} corresponds to Goussarov,
Polyak and Viro's definition in 
\cite{Goussarov/Polyak/Viro:FiniteTypeInvariants}.
\par
Writing $\widetilde{P}(\alpha,\tau,S,\nu,r)$ for $P(\alpha,\tau,S,r)$ modulo
shift moves and $\widetilde{G}(\alpha,\tau,S,\nu,r)$ for
$G(\alpha,\tau,S,r)$ modulo shift moves, it is easy to check that the
map $\theta_r$ induces an isomorphism from
$\Z\widetilde{P}(\alpha,\tau,S,\nu,r)$ to $\widetilde{G}(\alpha,\tau,S,\nu,r)$.
We then define $\widetilde{G}_n(\alpha,\tau,S,\nu,r)$ to be
$G_n(\alpha,\tau,S,r)$ modulo shift moves.
Thus $\Gamma_{n,r}$ induces a map from
$\Z\widetilde{P}(\alpha,\tau,S,\nu,r)$ to
$\widetilde{G}_n(\alpha,\tau,S,\nu,r)$ which, by analogous arguments to
those above, is a universal invariant of degree $n$.
\begin{rem}
In \cite{Goussarov/Polyak/Viro:FiniteTypeInvariants} the algebras
 $\mathcal{P}$ and $\mathcal{P}_n$ are defined.
Using the notation from Remark~\ref{rem:vknot-homotopy-closed}, we note
 that if
we consider $\mathcal{P}$ and $\mathcal{P}_n$ as additive groups (by
 forgetting about the multiplication operation), then
$\widetilde{G}(\alpha,\tau,S,\nu,1)$ is isomorphic to $\mathcal{P}$ and
$\widetilde{G}_n(\alpha,\tau,S,\nu,1)$ is isomorphic to $\mathcal{P}_n$.
\end{rem} 
Any invariant which is finite type of degree $n$ for a closed homotopy
is finite type of degree $n$ for the corresponding homotopy without
shift moves.
For degrees $0$ and $1$ the reverse is also true.
In other words $\widetilde{G}_i(\alpha,\tau,S,\nu,r)$ is equal to
$G_i(\alpha,\tau,S,r)$ for $i$ equal to $0$ or $1$ and for all $\alpha$,
$\tau$, $S$ and $\nu$.
For degree $2$ however, the reverse is not true.
For example, Fukunaga's $T$ invariant is not invariant under the shift 
move given by taking $\nu$ to be $\tau$, so it is not a finite type
invariant for the corresponding closed homotopy.
We also note that the invariant $v_4$ is not invariant under shift
moves.
\bibliography{mrabbrev,nanoftype}
\bibliographystyle{hamsplain}
\end{document}